\documentclass[11pt,a4paper]{article}

\usepackage{mathtools}
\usepackage{authblk} 
\usepackage{algpseudocode,algorithmicx,algorithm}
\usepackage{array}
\usepackage{booktabs}
\usepackage{etoolbox}
\usepackage{tcolorbox}
\usepackage[utf8]{inputenc}
\usepackage[T1]{fontenc}
\usepackage{graphicx}
\usepackage{amsthm}
\usepackage{amsmath}
\usepackage{mathdots}
\usepackage{verbatim}
\usepackage{pifont,url}
\usepackage{enumitem,amssymb}
\usepackage{multirow}
\usepackage{appendix}
\usepackage{epstopdf}
\usepackage{booktabs}
\usepackage{makecell}
\usepackage{diagbox}
\usepackage{setspace}
\usepackage{indentfirst}
\usepackage{qcircuit}

\usepackage{mathrsfs}
\usepackage{latexsym,bm}
\usepackage{amsmath,amsfonts,amsmath,amssymb,amsthm,bbm}
\usepackage{extarrows}

\usepackage{enumerate,cases,multirow}
\usepackage{subfigure}
\usepackage{longtable,colortbl,arydshln,threeparttable}
\definecolor{mygray}{gray}{.9}

\usepackage{indentfirst}
\usepackage{geometry}
\usepackage{cite}

\usepackage{listings}

\usepackage{makeidx}        
\usepackage{booktabs}
\usepackage[unicode=true,bookmarksnumbered,colorlinks,citecolor=red,linkcolor=red,hyperindex,linktocpage=true]{hyperref}

\def \d {\mathrm{d}}

\newcounter{parentalgorithm}

\makeatother

\newtheorem{theorem}{Theorem}[section]
\newtheorem{lemma}{Lemma}[section]

\theoremstyle{remark}
\newtheorem{remark}{\bf Remark}[section]

\numberwithin{equation}{section}
\usepackage{pifont}

\begin{document}

\title{Quantum Implicit-Explicit Schemes for Multiscale Ordinary and Partial Differential Equations via Schr\"odingerization}
\author[1,3]{Qitong Hu\thanks{huqitong@sjtu.edu.cn}$^\P$}
\author[1,2, 3]{Xiaoyang He\thanks{hexiaoyang@sjtu.edu.cn}$^\P$}
\author[1,2,3]{Shi Jin\thanks{shijin-m@sjtu.edu.cn}}
\author[1,3]{Xiao-Dong Zhang\thanks{xiaodong@sjtu.edu.cn}}
\affil[1]{School of Mathematical Sciences, Shanghai Jiao Tong University, Shanghai, 200240, China}
\affil[2]{Institute of Natural Sciences, Shanghai Jiao Tong University, Shanghai, 200240, China}
\affil[3]{Ministry of Education (MOE) Funded Key Lab of Scientific and Engineering Computing, Shanghai Jiao Tong University, Shanghai, 200240, China}

\date{}
\maketitle
\vspace{-2em}
\renewcommand{\thefootnote}{}
\footnotetext{\hspace{-0.4em}$^\P$These authors contributed equally to this work.}
\renewcommand{\thefootnote}{\arabic{footnote}}
\begin{abstract}
\par In this paper, we present a quantum implicit-explicit (IMEX) scheme for multiscale ordinary and partial differential equations whose discretization parameters are \textit{independent} of the scaling parameter $\varepsilon$. A key ingredient of our approach is a continuous-time formulation of classical IMEX schemes, which decouples the evolution time of the quantum algorithm from the physical time of the differential equation and is therefore particularly useful in multiscale settings. Building on this idea, we employ the Schr\"odingerization framework [Phys. Rev. Lett. 133 (2024), 230602] to implement IMEX schemes on quantum computers. Compared to previous HHL  type quantum AP scheme [J. Comput. Phys. 471 (2022), 111641], this new method requires narrower--an extra logarithmic factor-- auxiliary register numerical examples on linear heat and multiscale telegraph equations demonstrate the independence in $\varepsilon$ of the method. 
\end{abstract}
\textbf{Keywords}: Quantum IMEX Schemes, Schr\"odingerization Method, Asymptotic-Preserving Schemes, Multiscale Differential Equations.
\tableofcontents
\section{Introduction}
\par In the field of scientific computing, solving ordinary and partial differential equations via numerical methods is of great importance. However, when the number of degrees of freedom is sufficiently large, or when extremely high numerical resolution is required for multiscale problems or for large-scale simulations such as weather forecasting and turbulence, classical algorithms can face substantial computational challenges. In recent years, developing quantum algorithms to address scientific computing problems has attracted considerable interest. For example, 
Harrow, Hassidim, and Lloyd \cite{Harrow2009QuantumAF} introduced a quantum algorithm (the HHL algorithm) for solving large-scale linear algebraic systems, which can achieve exponential speedups under suitable assumptions; Berry \textit{et al.} \cite{Berry2015Hamiltonian} established a framework for quantum simulation \cite{Feynman1982Quantum} using Hamiltonian simulation.

\par At present, Hamiltonian simulation \cite{Lloyd1996UniversalQS,Berry2007Hamiltonians,Childs2010Quantum,Childs2012Hamiltonian} has demonstrated several quantum advantages in this area. For general ordinary differential equations (ODEs) and partial differential equations (PDEs), embedding non-Hamiltonian equations into the Hamiltonian-simulation framework is not only of significant scientific interest but also a theoretical challenge. Several approaches have been proposed recently to address this issue. The Schr\"odingerization method introduced by Jin \textit{et al.} \cite{Jin2024Schrodingerization} maps general linear ODEs and PDEs to higher-dimensional unitary evolutions that are suitable for quantum simulation, while a related method, the Linear Combination of Hamiltonian Simulation (LCHS) framework proposed by An \textit{et al.} \cite{An2023QuantumAF}, represents non-unitary dynamics as a linear combination of unitary operators.

\par In this paper, we focus on multiscale ODEs and PDEs, which are widely encountered in the physical sciences. These equations involve widely separated temporal or spatial scales, making it a central challenge to capture cross-scale interactions. In many cases, resolving the small physical scales numerically requires discretization parameters that depend on the small scaling parameter and therefore becomes prohibitively expensive on classical computers. Similarly, when Hamiltonian simulation is applied to multiscale dynamics such as $\frac{\d u(t)}{\d t}=\varepsilon^{-1}(L u(t)+b)$ with $t \in [0,T]$ and initial condition $u(0) = u_0$, the small parameter $\varepsilon$ remains a major bottleneck. Direct LCHS- or Schr\"odingerization-based algorithms, as well as their optimal improvements, still retain explicit $\varepsilon$-dependence in the query complexity \cite{An2023QuantumAF,Jin2024Schrodingerization,Low2025Optimal,Jin2025Optimal}. More recently, An \textit{et al.} \cite{An2026Fastforwarding} showed that dissipative ODEs with uniformly negative logarithmic norm can be fast-forwarded, obtaining polylogarithmic history-state complexity for a truncated Dyson-series method and $\sqrt{T}$-type final-state complexity in that dissipative setting. However, under the multiscale specialization relevant here, their final-state bounds still retain a $\sqrt{\varepsilon^{-1}}$-type dependence, and the forward-Euler-based variant also carries the target-accuracy factor $\delta^{-1}$. Meanwhile, HHL-based asymptotic-preserving approaches for specific multiscale PDEs can avoid explicit $\varepsilon$-dependence, but they rely on QLSA-type subroutines and related techniques that are significantly more complex to implement in practice \cite{Jin2022TimeCA}. The comparison is summarized in Tab.~\ref{table:1}.

\begin{table*}[htbp]
\centerline{
\resizebox{\textwidth}{!}{
\begin{tabular}{c|c|c|c|c}
\hline
\hline
    \makecell*[c]{Year}&
    \makecell*[c]{Reference}&
    \makecell*[c]{Query Complexity}&
    \makecell*[c]{Core Idea}&
    \makecell*[c]{Challenges}\\
\hline
\hline
    \makecell*[c]{2023}&
    \makecell*[c]{An \textit{et al.}\\ \cite{An2023QuantumAF}}&
    \makecell*[c]{$\mathcal{O}\left(\|L\|_{\max}\varepsilon^{-1}T\delta^{-1}\right)$}&
    \makecell*[c]{Direct application of linear combination\\ of Hamiltonian simulation (LCHS).}&
    \multirow{4}{*}{\makecell*[c]{\vspace{2em}\\ Query complexity depends explicitly\\ on the multiscale parameter\\ $\varepsilon^{-1}$.}}\\
\cline{1-4}
    \makecell*[c]{2024}&
    \makecell*[c]{Jin \textit{et al.}\\ \cite{Jin2024Schrodingerization}}&
    \makecell*[c]{$\mathcal{O}\left(\|L\|_{\max}\varepsilon^{-1}T\delta^{-1}\right)$}&
    \makecell*[c]{Direct application of Schr\"odingerization \cite{Jin2024Schrodingerization}.}&\\
\cline{1-4}
    \makecell*[c]{2025}&
    \makecell*[c]{Low \textit{et al.}\\  \cite{Low2025Optimal}}&
    \makecell*[c]{$\mathcal{O}\left(\|L\|_{\max}\varepsilon^{-1}T\log\delta^{-1}\right)$}&
    \makecell*[c]{Optimal LCHS.}&\\
\cline{1-4}
    \makecell*[c]{2025}&
    \makecell*[c]{Jin \textit{et al.}\\  \cite{Jin2025Optimal}}&
    \makecell*[c]{$\mathcal{O}\left(\|L\|_{\max}\varepsilon^{-1}T\log\delta^{-1}\right)$}&
    \makecell*[c]{Optimal Schr\"odingerization.}&\\
\hline
    \makecell*[c]{2026}&
    \makecell*[c]{An \textit{et al.}\\  \cite{An2026Fastforwarding}}&
    \makecell*[c]{$\mathcal{O}\left(\frac{\sqrt{\varepsilon^{-1} T}\|L\|_{\max}(\log\delta^{-1})^2}{\left(-\sup\limits_{t\in[0,T]}\lambda_{\max}\left(\frac{L(t)+L(t)^\dagger}{2}\right)\right)^{\frac{1}{2}}}\right)$}&
    \makecell*[c]{Based on truncated Dyson series.}&
    \multirow{2}{*}{\makecell*[c]{The final-state specialization used here still\\ depends on $\varepsilon^{-1}$, and the Euler-based\\ variant also retains a $\delta^{-1}$ factor.}}\\
\cline{1-4}
    \makecell*[c]{2026}&
    \makecell*[c]{An \textit{et al.}\\  \cite{An2026Fastforwarding}}&
    \makecell*[c]{$\mathcal{O}\left(\sqrt{\varepsilon^{-1} T}\delta^{-1}\right)$}&
    \makecell*[c]{Based on forward Euler method and\\ quantum linear system algorithms (QLSA).}&\\
\hline
    \makecell*[c]{2026}&
    \makecell*[c]{This paper}&
    \makecell*[c]{$\mathcal{O}\left(\frac{T\delta^{-1}(\log\delta^{-1})^3}{\left(-\sup\limits_{t\in[0,T]}\lambda_{\max}\left(\frac{L(t)+L(t)^\dagger}{2}\right)\right)^3}\right)$}&
    \makecell*[c]{Quantum implicit-explicit schemes\\ combined with the Schr\"odingerization method.}&
    \makecell*[c]{No $\varepsilon^{-1}$ dependence, so the query\\ complexity remains insensitive to stiffness\\ even when $\varepsilon\ll\delta$.}\\
\hline
\hline
\end{tabular}}}
\caption{\label{table:1}\textbf{Comparison of Quantum Algorithms for Solving Multiscale Ordinary and Partial Differential Equations.} We assume that the entries of $L(t)$ and $b(t)$ are of order $\mathcal{O}(1)$, and that the solution $u(t)$ is also of order $\mathcal{O}(1)$. $N_t$ is the number of time steps in the IMEX scheme and is independent of the multiscale parameter $\varepsilon$. For An \textit{et al.} \cite{An2026Fastforwarding}, the listed bounds correspond to final-state preparation specialized to the multiscale scaling in the dissipative setting considered there.}
\end{table*}

\par In classical computation, a popular and efficient strategy is to develop Asymptotic-Preserving (AP) schemes \cite{Jin2022AsymptoticpreservingSF}, which are designed to ensure that the numerical method automatically captures the correct macroscopic behavior even as the scale parameter $\varepsilon$ tends to zero, allowing numerical parameters--and hence computational complexity--to be \textit{independent} of $\varepsilon$. The first quantum AP scheme for multiscale PDEs was proposed in \cite{Jin2022TimeCA}, which gives an HHL-based quantum algorithm whose computational complexity is independent of $\varepsilon$. Implicit-Explicit (IMEX) schemes \cite{Ascher1995IMEX,Ascher1997IMEXRK,Sebastiano2023IM,Li2025IMEX} are simple yet effective AP schemes for time scaling and are widely used to address numerical stiffness and other multiscale challenges caused by small $\varepsilon$ in classical computation. This naturally raises the question of how IMEX schemes can be integrated into the Schr\"odingerization framework, which provides a systematic route to quantum simulation of PDEs and can achieve optimal or near-optimal complexity \cite{Jin2025LinearNonUnitary}.

\par In this article, we propose a general framework for quantum IMEX schemes that can be incorporated into Hamiltonian simulation, thereby providing a practical route to solving multiscale PDEs on quantum computers. Unlike traditional Schr\"odingerization-based methods for linear dynamical systems, we first use an IMEX scheme to transform the given ODE into a linear-system problem and then reformulate this linear system as a higher-dimensional ODE through continuous-time Richardson iteration. While this "ODE $\to$ Linear System $\to$ Equivalent ODE" transformation may seem counterintuitive, it yields an equivalent ODE whose query complexity is \textit{independent} of the scaling parameter $\varepsilon$ (see Theorem \ref{theorem:main}) and provides a general framework for multiscale equations.

\par To validate the theory, we present simulations for two representative PDEs, namely heat equations with stiff terms and multiscale telegraph equations. Our advantage over QLSA-based AP solvers is not query complexity but hardware overhead. For HHL-type implementations, the main auxiliary-register overhead comes from the phase-estimation register. A recent HHL analysis makes this dependence explicit: if the clock register has $n_c$ qubits, then its maximum eigenvalue-estimation precision is $2^{-n_c}$ \cite{Ginzburg2024QLSA}; therefore, achieving phase-estimation accuracy $\delta$ requires $n_c=\mathcal{O}(\log\delta^{-1})$. In addition, HHL uses one reciprocal-eigenvalue ancilla and only $\mathcal{O}(1)$ further work qubits \cite{Harrow2009QuantumAF}. Precision-improved VTAA/RM variants still require additional control ancillas and subroutines \cite{Childs2017SIAM,Subasi2019PRL}. Although neither class of algorithms is fully NISQ-ready, a narrower auxiliary register is still preferable on width-limited devices and in small-scale proof-of-principle demonstrations, since it reduces the extra qubit footprint and avoids reciprocal-eigenvalue subroutines. By contrast, our Schr\"odingerized IMEX implementation only needs the auxiliary $p$-register introduced later and $\mathcal{O}(1)$ additional ancillas. In the discrete Schr\"odingerization framework one has $N_p=2^{n_p}$ Fourier modes in the $p$ direction, so the register width is $n_p=\log_2 N_p$. More precisely, the optimal smooth-initialization analysis of \cite{Jin2025Optimal} yields a logarithmic bound for the largest Fourier mode cutoff, and in the present normalization $\mu_\ell=\pi \ell$ this means $\|D_\mu\|_{\max}=\Theta(N_p)=\mathcal{O}(\log\delta^{-1})$; therefore $n_p=\mathcal{O}(\log\log\delta^{-1})$. For the polynomially conditioned linear systems arising from the PDE discretizations considered here, the present approach reduces the auxiliary-register width by at least a logarithmic factor in $N_x$ and avoids reciprocal-eigenvalue subroutines. A comparison with existing quantum algorithms for these two equations is summarized in Tab.~\ref{table:2}.

\begin{table*}[htbp]
\centerline{
\resizebox{\textwidth}{!}{
\begin{tabular}{c|c|c|c|c|c|c}
\hline
\hline
    \makecell*[c]{Year}&
    \makecell*[c]{Reference}&
    \makecell*[c]{Equation}&
    \makecell*[c]{Query Complexity}&
    \makecell*[c]{Auxiliary Register\\ Width}&
    \makecell*[c]{Core Idea}&
    \makecell*[c]{Remarks}\\
\hline
\hline
    \multirow{2}{*}{\makecell*[c]{\vspace{0em}\\-}}&
    \multirow{2}{*}{\makecell*[c]{\vspace{-0.5em}\\Classical IMEX\\ baseline}}&
    \makecell*[c]{Heat}&
    \makecell*[c]{$\mathcal{O}(N_x^3)$}&
    \makecell*[c]{N/A}&
    \multirow{2}{*}{\makecell*[c]{\vspace{0em}\\Classical IMEX discretization with\\ a sparse linear solve at each step.}}&
    \multirow{2}{*}{\makecell*[c]{\vspace{0em}\\Higher asymptotic cost than the quantum methods listed below.}}\\
\cline{3-5}
    &&
    \makecell*[c]{Multiscale\\ Telegraph}&
    \makecell*[c]{$\mathcal{O}(N_x^3)$}&
    \makecell*[c]{N/A}&\\
\hline
    \makecell*[c]{2022}&
    \makecell*[c]{Jin \textit{et al.}\\ \cite{Jin2022TimeCA}}&
    \makecell*[c]{Heat}&
    \makecell*[c]{$\mathcal{O}(N_x^2\log N_x)$}&
    \makecell*[c]{$\mathcal{O}(\log N_x)$}&
    \multirow{2}{*}{\makecell*[c]{\vspace{-0.5em}\\AP finite-difference discretization\\ solved by an HHL-type QLSA.}}&
    \multirow{2}{*}{\makecell*[c]{\vspace{-1.8em}\\Competitive query complexity, but the eigenvalue-estimation,\\ reciprocal-eigenvalue, and VTAA/RM subroutines increase\\ circuit width and implementation complexity\\ \cite{Childs2017SIAM,Subasi2019PRL,Jin2025Precondition}.}}\\
\cline{1-5}
    \makecell*[c]{2022/2023}&
    \makecell*[c]{Jin \textit{et al.} \cite{Jin2022TimeCA}\\ He \textit{et al.} \cite{He2023TimeCA}}&
    \makecell*[c]{Multiscale\\ Telegraph}&
    \makecell*[c]{$\mathcal{O}(N_x^2\log N_x)$}&
    \makecell*[c]{$\mathcal{O}(\log N_x)$}&\\
\hline
    \makecell*[c]{2024/2025}&
    \makecell*[c]{Jin \textit{et al.}\\ \cite{Jin2024Schrodingerization,Jin2025Optimal}}&
    \makecell*[c]{Heat}&
    \makecell*[c]{$\mathcal{O}\left(\varepsilon^{-1}N_x^2(\log N_x)^2\right)$}&
    \makecell*[c]{$\mathcal{O}\left(\log\log N_x\right)$}&
    \makecell*[c]{\makecell*[c]{Direct Schr\"odingerization, with the optimal\\ smooth-initialization refinement of \cite{Jin2025Optimal}.}}&
    \makecell*[c]{\makecell*[c]{Narrower auxiliary-register width than HHL-type QLSA methods,\\ but the query complexity still depends explicitly on $\varepsilon^{-1}$.}}\\
\hline
    \multirow{2}{*}{\makecell*[c]{\vspace{0em}\\2025}}&
    \multirow{2}{*}{\makecell*[c]{\vspace{0em}\\This paper}}&
    \makecell*[c]{Heat}&
    \makecell*[c]{$\mathcal{O}(N_x^2(\log N_x)^3)$}&
    \makecell*[c]{$\mathcal{O}(\log\log N_x)$}&
    \multirow{2}{*}{\makecell*[c]{AP IMEX discretization followed by\\ Schr\"odingerization of the equivalent\\ continuous-time Richardson system.}}&
    \multirow{2}{*}{\makecell*[c]{No explicit $\varepsilon^{-1}$ dependence; narrower auxiliary\\ register than HHL-based AP solvers; avoids reciprocal-eigenvalue\\ subroutines; nearly optimal complexity up to logarithmic factors.}}\\
\cline{3-5}
    &&
    \makecell*[c]{Multiscale\\ Telegraph}&
    \makecell*[c]{$\mathcal{O}(N_x^{2+o(1)}(\log N_x)^3)$}&
    \makecell*[c]{$\mathcal{O}\left(\log\log N_x\right)$}&\\
\hline
\hline
\end{tabular}}}
\caption{\label{table:2}\textbf{Comparison of Existing Algorithms for Solving Heat Equations and the Multiscale Telegraph Equation.} Under the same assumptions as in Tab.~\ref{table:1}. In the query-complexity column we suppress fixed dependence on the spatial dimension; for the Schr\"odingerization-based rows, the remaining logarithmic accuracy factors are rewritten using the common identification $N_x=\delta^{-1}$. The classical baseline counts one sparse linear solve per IMEX time step. The auxiliary-register-width column records the dominant extra qubit overhead beyond the system register and is also written in terms of $N_x$: it is $\mathcal{O}(\log N_x)$ for the HHL/AP rows from phase estimation and $\mathcal{O}(\log\log N_x)$ for the direct Schr\"odingerization rows from the logarithmic Fourier-mode cutoff in \cite{Jin2025Optimal}.}
\end{table*}

\par The rest of this paper is organized as follows. In Section 2, we review the basic Schr\"odingerization method, focusing on the algorithm for linear time-independent systems. In Sections 3 and 4, we present our quantum IMEX schemes and their specific procedures, including the algorithmic framework and query-complexity estimates. In Section 5, we provide discretization schemes and simulation results for specific multiscale PDEs. In the Appendix, we present proofs of the main theorems.

\section{Review of the Schr\"odingerization Method}
\label{section:schr}
\par In this section, we briefly review the Schr\"odingerization method for linear dynamical systems with time-independent constant coefficients:
\begin{gather}
    \nonumber\frac{\d u(t)}{\d t}=-Hu(t),\quad t\in[0,T],\\
    \text{with }u(0)=u_0,
    \label{equ:schrodingerization:1}
\end{gather}
where $H$ is a constant matrix of size $N\times N$ with $\lambda_{\min}\left(\frac{H+H^\dagger}{2}\right)>0$. The matrix $H$ can be decomposed into the sum of a Hermitian matrix and an anti-Hermitian matrix as follows:
\begin{equation*}
    \begin{aligned}
    \label{equ:schrodingerization:2}
        H=H_1+iH_2,\quad H_1=\frac{H+H^\dagger}{2},\quad H_2=\frac{H-H^\dagger}{2i}.
    \end{aligned}
\end{equation*}
We apply the warped phase transformation $u_{\text{warp}}(t,p)=e^{-p}u(t)$, in which $p>0$ (and can be symmetrically extended to $p<0$). This allows Eq.~(\ref{equ:schrodingerization:1}) to be rewritten as
\begin{gather}
    \nonumber\frac{\d u_{\text{warp}}(t,p)}{\d t}=-H_1\partial_p u_{\text{warp}}(t,p)-iH_2u_{\text{warp}}(t,p)\text{, $t\in[0,T]$,}\\
    \text{with }u_{\text{warp}}(0,p)=e^{-\vert p\vert}u(0),
    \label{equ:schrodingerization:3}
\end{gather}
which is hyperbolic. This lack of regularity in the initial data can be remedied. See \cite{Jin2025LinearNonUnitary} for a choice of smooth initial data in $p$ that yields nearly optimal or even optimal complexity.

\subsection{The Discrete Schr\"odingerization Method}

\par We apply the discrete Fourier transform to $p$ on the interval $[L,R]$, discretized as $L=p_0<p_1<\cdots<p_{N_p}=R$, where $\Delta p=\frac{R-L}{N_p}$ and $p_k=L+k\Delta p$. Here $L$ and $R$ should be chosen sufficiently large so that the solution is negligible at the endpoints of the $p$-domain. Define the vector $u_{\text{Four}}(t)$ as
\begin{equation*}
    \begin{aligned}
        u_{\text{Four},i}(t)=\sum\limits_{k=0}^{N_p-1} u_{\text{warp},i}(t,p_k)|k\rangle,\quad u_{\text{Four}}(t) = [u_{\text{Four},1}(t);\cdots;u_{\text{Four},N}(t)],
    \end{aligned}
\end{equation*}
in which both $u_{\text{Four},i}(t)$ and $u_{\text{warp},i}(t)$ represent the values in the $i$-th component of their corresponding vectors. In the discrete Fourier space, one obtains
\begin{gather}
    \nonumber\frac{\d}{\d t}u_{\text{Four}}(t)=-i(H_1\otimes P_{\mu})u_{\text{Four}}(t)-i(H_2\otimes I)u_{\text{Four}}(t),\\
    u_{\text{Four}}(0) = [e^{-|p_0|};\cdots;e^{-|p_{N_p-1}|}] \otimes u_0,
    \label{equ:schrodingerization:4}
\end{gather}
where $P_{\mu}$ denotes the discrete momentum operator $-i\partial_p$ expressed in matrix form using discretization in space. The diagonalization of $P_{\mu}$ is achieved through the transformation $D_{\mu}=\phi^{-1}P_{\mu}\phi$, where $D_{\mu}$ is a diagonal matrix with entries $\mu_{-N_p/2}$ through $\mu_{N_p/2-1}$, and $\phi$ is defined as $\phi_{j\ell}=\phi_\ell(p_j)$ with $\phi_\ell(p)=e^{i\mu_\ell(p-L)}$. These diagonal elements are defined at $\mu_{\ell}=\pi \ell$ for integer values $\ell$ ranging from $-N_p/2$ to $N_p/2-1$. Applying the variable transformation $u_{\text{schr}}=(I\otimes\phi^{-1})u_{\text{Four}}$, we have
\begin{gather}
    \nonumber\frac{\d}{\d t}u_{\text{schr}}(t)=-i(H_1\otimes D_{\mu}+H_2\otimes I)u_{\text{schr}}(t):=-iH_{\text{schr}}\cdot u_{\text{schr}}(t),\\
    \text{with }u_{\text{schr}}(0)=(\phi^{-1}\otimes I)u_{\text{Four}}(0).
    \label{equ:schrodingerization:5}
\end{gather}
Through the Schr\"odingerization method, one transforms the homogeneous equation in Eq.~(\ref{equ:schrodingerization:1}) into the dimension-lifted equation shown in Eq.~(\ref{equ:schrodingerization:5}), where $H_{\text{schr}}$ is a Hermitian matrix. This enables us to simulate the system in Eq.~(\ref{equ:schrodingerization:5}) on a quantum computer. 

\subsection{Reconstruction of the Solution}
\par If the eigenvalues of $H_1$ are all non-positive, then $u(t)$ can be recovered from $u_{\text{warp}}$ by the following two methods: the single-point method and the integral method:
\begin{equation*}
    \begin{aligned}
    &\text{The single-point method: }u(t)=e^{p_R}u_{\text{warp}}(t,p_R), \quad \text{for any } p_R>0, \\
    &\text{The integral method: }u(t)=\frac{1}{e^{p_R}-1}\int_0^{p_R} u_{\text{warp}}(t,q)\d q.
    \end{aligned}
\end{equation*}
However, if $H_1$ contains positive eigenvalues, then spurious solutions may appear in the region $p>0$ \cite{Jin2024IllPosed}. Therefore, when reconstructing $u(t)$, the correct domain must be selected to avoid them. As established in \cite{Jin2025LinearSystems}, the following result holds.
\begin{theorem}
\par  If the largest eigenvalue of $H_1$ is positive, denoted by $\lambda_1(H_1)>0$, while the remaining eigenvalues are ordered as $\lambda_1(H_1)\ge\lambda_2(H_1)\ge\cdots\ge\lambda_N(H_1)$, the solution to Eq.~(\ref{equ:schrodingerization:1}) can be reconstructed by
\begin{equation}
    \begin{aligned}
    \label{equ:schrodingerization:6}
        u(t)=e^p u_{\text{warp}}(t,p)\text{, for any $p>p^\Diamond$},
    \end{aligned}
\end{equation}
where $p^\Diamond=\max\left\{\lambda_1(H_1)T,0\right\}$, or 
\begin{equation}
    \begin{aligned}
    \label{equ:schrodingerization:7}
        u(t)=e^p\int_p^{\infty} u_{\text{warp}}(t,q)\d q\text{, for any $p>p^\Diamond$}.
    \end{aligned}
\end{equation}
\end{theorem}

\section{Quantum IMEX Schemes for Equations with Time-Dependent Coefficients}
\label{section:timede}

\par In this section, we present the quantum IMEX schemes in detail. Consider the dynamical system with time-dependent coefficients and initial condition:
\begin{gather}
    \label{equ:structure:2}\frac{\d u(t)}{\d t}=L(t)u(t)+b(t),\quad t\in[0,T],\quad u(0)=u_0.
\end{gather}
in which $L(t)$ and $b(t)$ are time-dependent matrices of sizes $N_x\times N_x$ and $N_x\times 1$, respectively, and $L(t)$ can be decomposed as $L(t) = \varepsilon^{-1} L_1(t) + L_2(t)$, where $\varepsilon^{-1} L_1(t)$ represents the stiff terms with $\varepsilon$ being the scaling parameter, and $L_2(t)$ represents the non-stiff terms. Similarly, $b(t)$ can be decomposed into $\varepsilon^{-1} b_1(t)$ and $b_2(t)$, with $L_1(t)$, $L_2(t)$, $b_1(t)$, $b_2(t)$, and $u_0$ being of order $\mathcal{O}(1)$. A simple implicit-explicit (IMEX) scheme is then applied:
\begin{equation*}
    \begin{aligned}
        \frac{u_{n+1}-u_n}{\tau} = \varepsilon^{-1}L_1((n+1)\tau)u_{n+1}+L_2(n\tau)u_n+\varepsilon^{-1}b_1(n\tau)+b_2(n\tau),
    \end{aligned}
\end{equation*}
that is
\begin{equation}
    \begin{aligned}
        \label{equ:structure:1}
        P_n u_{n+1}=Q_n u_n&+b_n\text{, }n=0,1\cdots,
    \end{aligned}
\end{equation}
where $P_{n}=\varepsilon I-\tau L_1((n+1)\tau)$, $Q_n=\varepsilon (I+\tau L_2(n\tau))$, and $b_n=\tau b_1(n\tau)+\tau\varepsilon b_2(n\tau)$; these quantities are introduced for convenience in the subsequent analysis.
Our algorithm requires that the spectra of $P_n$ and $Q_n$ satisfy certain conditions, and the specific requirements are given by the constraints in Lemma \ref{lemma:main:1}. Generally, when the scaling parameter $\varepsilon \to 0$, the condition reduces to $\sup\limits_{t\in[0,T]}\lambda_{\max}\left(\frac{L_1(t)+L_1(t)^\dagger}{2}\right)<0$, meaning that only $P_n$, which contains the stiff term, is restricted.

\subsection{Estimation for Discretization Numbers}
\label{section:discretization}
\par We now estimate the number of time steps required by the IMEX discretization in Eq.~(\ref{equ:structure:1}). The goal is to determine a sufficient resolution $N_t$ such that the relative global error at the final time $T$ remains below a prescribed tolerance $\delta>0$. Let $t_n=n\tau$ with $\tau=T/N_t$, and define the global error by $e_n=u(t_n)-u_n$, where $u(t_n)$ denotes the exact solution of the continuous problem (\ref{equ:structure:2}) and $u_n$ the numerical approximation. We seek $N_t$ such that 
\begin{equation}
    \begin{aligned}
        \label{equ:euler:3}
        \frac{\|u(T)-u_{N_t}\|}{\|u(T)\|}=\frac{\|e_{N_t}\|}{\|u(T)\|}<\delta,
    \end{aligned}
\end{equation}
which requires a sharp estimate of the accumulated error $e_{N_t}$ for the IMEX scheme.

\par Substituting the exact solution into the discrete scheme of Eq.~(\ref{equ:structure:1}), we define the local defect (or truncation error) as
\begin{equation*}
    \begin{aligned}
        \eta_n:=P_nu(t_{n+1})-Q_nu(t_n)-b_n,\quad n=0,\cdots,N_t-1.
    \end{aligned}
\end{equation*}
This quantity quantifies the extent to which the exact solution fails to satisfy the numerical update formula at a single step. Since the initial data are exact, $e_0=0$, the propagation of the global error obeys the linear recurrence
\begin{equation*}
    \begin{aligned}
        P_n e_{n+1}=Q_n e_n+\eta_n,\qquad
        e_{n+1}=P_n^{-1}Q_n e_n+P_n^{-1}\eta_n.
    \end{aligned}
\end{equation*}
Thus, the evolution of the error is governed by the amplification matrix $P_n^{-1}Q_n$ and the local perturbation $P_n^{-1}\eta_n$.

\par To bound $\eta_n$, we add and subtract $\tau\varepsilon u^\prime(t_{n+1})$ and then employ the governing Eq.~(\ref{equ:structure:2}). A straightforward algebraic manipulation yields
\begin{equation*}
    \begin{aligned}
        \eta_n
        &=\varepsilon\bigl[u(t_{n+1})-u(t_n)-\tau u^\prime(t_{n+1})\bigr] +\tau\varepsilon\bigl[L_2(t_{n+1})u(t_{n+1})-L_2(t_n)u(t_n)\bigr]\\
        &\quad +\tau\bigl[b_1(t_{n+1})-b_1(t_n)\bigr]+\tau\varepsilon\bigl[b_2(t_{n+1})-b_2(t_n)\bigr].
    \end{aligned}
\end{equation*}
This decomposition isolates the temporal discretization errors associated with the stiff and non-stiff components. Using Taylor expansion with integral remainder and the mean-value theorem, each bracketed difference can be bounded by the corresponding first or second derivative of the involved functions. Consequently, we obtain the local error estimate
\begin{equation}
    \begin{aligned}
        \label{equ:euler:4}
        \|\eta_n\|\le \Bigl(\frac{\varepsilon}{2}\|u^{\prime\prime}\|_{\max}+\varepsilon\|(L_2u)^\prime\|_{\max}+\|b_1^\prime\|_{\max}+\varepsilon\|b_2^\prime\|_{\max}\Bigr)\tau^2,
    \end{aligned}
\end{equation}
where the notation $\|f\|_{\max}=\sup_{t\in[0,T]}\|f(t)\|$ denotes the maximal norm of a function over the time interval. We emphasize that the local error is of order $\mathcal{O}(\tau^2)$ with constants that exhibit only weak dependence on the stiffness parameter $\varepsilon$; in particular, the stiff part of the source term $b_1$ enters without an inverse power of $\varepsilon$, which is crucial for the subsequent uniform-in-$\varepsilon$ convergence analysis.

\par For the global error to remain controlled, the iteration matrices must not amplify perturbations excessively. Under the spectral condition $\sup_{t\in[0,T]}\lambda_{\max}\Bigl(\frac{L_1(t)+L_1(t)^\dagger}{2}\Bigr)<0$, which ensures that the stiff linear part is strictly dissipative, and assuming $\varepsilon$ is sufficiently small, one can derive explicit bounds on the operator norms of $P_n^{-1}$ and $P_n^{-1}Q_n$. Specifically,
\begin{equation*}
    \begin{aligned}
        \|P_n^{-1}\|_2&\le\frac{1}{\varepsilon-\tau\sup\limits_{t\in[0,T]}\lambda_{\max}\bigl(\frac{L_1(t)+L_1(t)^\dagger}{2}\bigr)},\quad 
        \|P_n^{-1}Q_n\|_2&\le\frac{\varepsilon\bigl(1+\tau\sup_{t\in[0,T]}\|L_2(t)\|_2\bigr)}{\varepsilon-\tau\sup\limits_{t\in[0,T]}\lambda_{\max}\bigl(\frac{L_1(t)+L_1(t)^\dagger}{2}\bigr)}<1.
    \end{aligned}
\end{equation*}
The strict inequality for $\|P_n^{-1}Q_n\|_2$ guarantees that the error propagation is contractive, preventing the accumulation of local defects. Iterating the one-step error recurrence and employing the geometric series summation, we arrive at the global error estimate
\begin{equation}
    \begin{aligned}
        \label{equ:euler:5}
        \|e_{N_t}\|
        &\le \sum_{k=0}^{N_t-1}\Bigl(\max_{n}\|P_n^{-1}Q_n\|_2\Bigr)^{N_t-1-k}\|P_k^{-1}\|_2\,\|\eta_k\|\\
        &\le \frac{\max\limits_{n}\|P_n^{-1}\|_2\bigl(\frac{\varepsilon}{2}\|u^{\prime\prime}\|_{\max}+\varepsilon\|(L_2u)^\prime\|_{\max}+\|b_1^\prime\|_{\max}+\varepsilon\|b_2^\prime\|_{\max}\bigr)\tau^2}{1-\max\limits_{n}\|P_n^{-1}Q_n\|_2}\\
        &\lesssim \frac{\bigl(\frac{\varepsilon}{2}\|u^{\prime\prime}\|_{\max}+\varepsilon\|(L_2u)^\prime\|_{\max}+\|b_1^\prime\|_{\max}+\varepsilon\|b_2^\prime\|_{\max}\bigr)\tau}{-\sup\limits_{t\in[0,T]}\lambda_{\max}\bigl(\frac{L_1(t)+L_1(t)^\dagger}{2}\bigr)},
    \end{aligned}
\end{equation}
where in the final step we have retained the dominant term as $\varepsilon\to0$. This inequality reveals that the global IMEX error is \emph{first order} in the time step $\tau$, and, more importantly, that the error constant is uniformly bounded with respect to $\varepsilon$.

\par To satisfy the prescribed tolerance $\|e_{N_t}\|<\delta\|u(T)\|$, it is sufficient to enforce
\begin{equation}
    \begin{aligned}
        \label{equ:euler:6}
        \tau\lesssim -\sup_{t\in[0,T]}\lambda_{\max}\Bigl(\frac{L_1(t)+L_1(t)^\dagger}{2}\Bigr)\frac{\delta\|u(T)\|}{\frac{\varepsilon}{2}\|u^{\prime\prime}\|_{\max}+\varepsilon\|(L_2u)^\prime\|_{\max}+\|b_1^\prime\|_{\max}+\varepsilon\|b_2^\prime\|_{\max}}.
    \end{aligned}
\end{equation}
Consequently, the number of time steps $N_t = T/\tau$ must satisfy the lower bound
\begin{equation}
    \begin{aligned}
        \label{equ:euler:7}
        N_t\gtrsim \frac{T}{\delta}\cdot\mathcal{O}\Biggl(\frac{\frac{\varepsilon}{2}\|u^{\prime\prime}\|_{\max}+\varepsilon\|(L_2u)^\prime\|_{\max}+\|b_1^\prime\|_{\max}+\varepsilon\|b_2^\prime\|_{\max}}{-\sup\limits_{t\in[0,T]}\lambda_{\max}\bigl(\frac{L_1(t)+L_1(t)^\dagger}{2}\bigr)\,\|u(T)\|}\Biggr).
    \end{aligned}
\end{equation}
In particular, if the quantities appearing in the order constant remain of size $\mathcal{O}(1)$ as $\varepsilon\to 0$, then this lower bound is \emph{independent} of the multiscale parameter $\varepsilon$. This is a hallmark of an asymptotic-preserving (AP) time integrator: the step size can be chosen solely based on accuracy requirements for the macroscopic dynamics, without any degradation due to stiffness. By contrast, a fully explicit discretization would still be subject to a CFL-type restriction of the form $\tau \le C\varepsilon$, leading to $N_t \propto 1/\varepsilon$ and rendering long-time simulations prohibitively expensive when $\varepsilon\ll1$.

\subsection{Configuration and Computation for the Quantum IMEX Schemes}

\subsubsection{Structure of the Quantum IMEX Schemes}

\par Next, we formulate the quantum IMEX scheme corresponding to the iterative form in Eq.~(\ref{equ:structure:1}). Let $\mathbf{u}=[u_{N_t};\cdots;u_1]$; then one obtains the linear system for $\mathbf{u}$ as follows:
\begin{equation}
    \begin{aligned}
        \label{equ:system:1}
        \mathbf{H}\mathbf{u}=\mathbf{F},
    \end{aligned}
\end{equation}
in which
\begin{equation}
    \begin{aligned}
        \label{equ:system:2}
        \mathbf{H}=
        \begin{bmatrix}
        P_{N_t-1} & -Q_{N_t-1} &&\\
        & P_{N_t-2} & -Q_{N_t-2} &\\
        && \ddots & \ddots \\
        &&& P_1 & -Q_1 \\
        &&&& P_0\\
        \end{bmatrix},\quad
        \mathbf{F}=
        \begin{bmatrix}
        b_{N_t-1}\\
        b_{N_t-2}\\
        \vdots\\
        b_1\\
        Q_0 u_0+b_0
        \end{bmatrix}.
    \end{aligned}
\end{equation}
To solve the linear system in Eq.~(\ref{equ:system:1}), we apply the continuous-time iteration method and seek the steady-state solution of the following ODE:
\begin{equation}
    \begin{aligned}
        \label{equ:system:3}
        \frac{\d\mathbf{u}(t)}{\d t}=\mathbf{F}-\mathbf{H}\mathbf{u}(t)\text{, and $\mathbf{u}(0)=\mathbf{u}_0$ is given.}
    \end{aligned}
\end{equation}
Under the spectral condition that each $P_n$ has eigenvalues with positive real parts, $\mathbf{H}$ is block upper triangular with diagonal blocks $P_n$, so the eigenvalues of $-\mathbf{H}$ have negative real parts and the continuous-time iteration converges to a steady state.
\subsubsection{The Schr\"odingerization Method for Solving the Quantum IMEX Schemes}
\label{section:schr:evol}

\par To perform Hamiltonian simulation for Eq.~(\ref{equ:system:3}), we use the Schr\"odingerization method for dynamical systems with time-dependent coefficients. We first convert Eq.~(\ref{equ:system:3}) into the following homogeneous form:
\begin{equation}
    \begin{aligned}
        \label{equ:system:4}
        \frac{\d\mathbf{u}_{\text{homo}}(t)}{\d t}=
        -\mathbf{H}_{\text{homo}}
        \mathbf{u}_{\text{homo}}(t),
    \end{aligned}
\end{equation}
in which $\mathbf{H}_{\text{homo}}=\begin{bmatrix}\mathbf{H} & -\text{diag}(\mathbf{F})\\ O & O\end{bmatrix}$, where $O$ denotes the zero matrix of appropriate size, and $\mathbf{u}_{\text{homo}}(t)=[\mathbf{u}(t);\mathbf{1}]$. The initial condition is given by $\mathbf{u}_{\text{homo}}(0)=[\mathbf{u}_0;\mathbf{1}]$, where $\mathbf{1}$ denotes the vector with all components equal to $1$, with dimension matching that of $\mathbf{F}$. We then decompose $\mathbf{H}_{\text{homo}}$ into its Hermitian and anti-Hermitian components:
\begin{equation*}
    \begin{aligned}
        \mathbf{H}_{\text{homo}}=\mathbf{H}_{\text{homo},1}+ i\mathbf{H}_{\text{homo},2},\quad
        \mathbf{H}_{\text{homo},1}=\frac{\mathbf{H}_{\text{homo}}+\mathbf{H}_{\text{homo}}^{\dagger}}{2},\quad 
        \mathbf{H}_{\text{homo},2}=\frac{\mathbf{H}_{\text{homo}}-\mathbf{H}_{\text{homo}}^{\dagger}}{2i}.
    \end{aligned}
\end{equation*}
To avoid conflict with the Hermitian part $\frac{\mathbf{H}+\mathbf{H}^{\dagger}}{2}$ used later in the decay analysis, we keep the notation $\mathbf{H}_{\text{homo},1}$ and $\mathbf{H}_{\text{homo},2}$ explicit throughout this subsection.
By applying the warped phase transformation $\mathbf{u}_{\text{warp}}(t,p)=e^{-p}\mathbf{u}_{\text{homo}}(t)$ for $p>0$, and symmetrically extending the initial data to $p<0$, Eq.~(\ref{equ:system:4}) is transformed into a system of linear convection equations:
\begin{equation}
    \begin{aligned}
        \label{equ:system:5}
        \frac{\partial\mathbf{u}_{\text{warp}}(t,p)}{\partial t}&=-\mathbf{H}_{\text{homo},1}\frac{\partial\mathbf{u}_{\text{warp}}(t,p)}{\partial p}-i\mathbf{H}_{\text{homo},2}\mathbf{u}_{\text{warp}}(t,p),\\
        \mathbf{u}_{\text{warp}}(0,p)&=e^{-\vert p\vert}\mathbf{u}_{\text{homo}}(0).
    \end{aligned}
\end{equation}
Applying the same discrete Fourier transform method introduced in Section \ref{section:schr}, and defining $\mathbf{u}_{\text{schr}}(t)=\left[\mathcal{F}[\mathbf{u}_{\text{warp}}(t,-\frac{N_p}{2}+1)];\cdots;\mathcal{F}[\mathbf{u}_{\text{warp}}(t,\frac{N_p}{2})]\right]$, we can transform Eq.~(\ref{equ:system:5}) into:
\begin{equation}
    \begin{aligned}
        \label{equ:system:6}
        \frac{\d\mathbf{u}_{\text{schr}}(t)}{\d t}
        &=-i(D_{\mu}\otimes \mathbf{H}_{\text{homo},1}+I\otimes \mathbf{H}_{\text{homo},2})\mathbf{u}_{\text{schr}}(t):=-i\mathbf{H}_{\text{schr}}\mathbf{u}_{\text{schr}}(t),
    \end{aligned}
\end{equation}
where the definition of $D_{\mu}$ is the same as in Section \ref{section:schr}.
\subsection{Decay toward the steady state  
}
\par Before proceeding with the calculations, we estimate the decay rate toward the steady state of the ODE in Eq.~(\ref{equ:system:3}). Under the additional assumption that $\mathbf{H}$ is Hermitian positive definite, a standard estimate, also used in \cite{Hu2024QuantumMultiscale}, is the following.


\begin{lemma}
\label{lemma:error:1}
\par For the ODE presented in Eq.~(\ref{equ:system:3}), if $\mathbf{H}$ is positive-definite Hermitian and its eigenvalues satisfy $\lambda_{\max}(\mathbf{H})\ge\cdots\ge\lambda_{\min}(\mathbf{H})>0$, then $\mathbf{u}(t)$ converges to the steady state $\mathbf{u}_{\infty}$ as follows:
\begin{equation}
    \begin{aligned}
        \label{equ:error:1}
        \Vert \mathbf{u}(t)-\mathbf{u}_{\infty}\Vert_2\le e^{-\lambda_{\min}(\mathbf{H})t}\Vert \mathbf{u}_0-\mathbf{u}_{\infty}\Vert_2.
    \end{aligned}
\end{equation}
\qed
\end{lemma}
\noindent The next lemma removes the symmetry condition on $\mathbf{H}$ and replaces the positive-definiteness requirement on $\mathbf{H}$ with one on $\mathbf{H}_1=\frac{\mathbf{H}+\mathbf{H}^{\dagger}}{2}$. 
We state and prove this lemma below:
\begin{lemma}
\label{lemma:error:2}
\par For the ODE presented in Eq.~(\ref{equ:system:3}), if $\mathbf{H}_1=\frac{\mathbf{H}+\mathbf{H}^{\dagger}}{2}$ is positive-definite Hermitian and its eigenvalues satisfy $\lambda_{\max}(\mathbf{H}_1)\ge\cdots\ge\lambda_{\min}(\mathbf{H}_1)>0$, then $\mathbf{u}(t)$ converges to the steady state $\mathbf{u}_{\infty}$ as 
\begin{equation}
    \begin{aligned}
        \label{equ:error:2}
        \Vert \mathbf{u}(t)-\mathbf{u}_{\infty}\Vert_2\le e^{-\lambda_{\min}(\mathbf{H}_1)t}\Vert \mathbf{u}_0-\mathbf{u}_{\infty}\Vert_2.
    \end{aligned}
\end{equation}
\begin{proof}
\par We first derive the ODE for $[\mathbf{u}(t)-\mathbf{u}_{\infty}]$ by subtracting $\frac{\d\mathbf{u}_{\infty}}{\d t}$ from $\frac{\d\mathbf{u}(t)}{\d t}$:
\begin{equation}
    \begin{aligned}
        \label{equ:error:3}
        \frac{\d[\mathbf{u}(t)-\mathbf{u}_{\infty}]}{\d t}=-\mathbf{H}[\mathbf{u}(t)-\mathbf{u}_{\infty}].
    \end{aligned}
\end{equation}
Then, we consider the derivative of $\Vert \mathbf{u}(t)-\mathbf{u}_{\infty}\Vert_2^2$, and obtain
\begin{equation}
    \begin{aligned}
        \label{equ:error:4}
        \frac{\d\Vert \mathbf{u}(t)-\mathbf{u}_{\infty}\Vert_2^2}{\d t}=\left(\frac{\d[\mathbf{u}(t)-\mathbf{u}_{\infty}]^\dagger}{\d t}[\mathbf{u}(t)-\mathbf{u}_{\infty}]+[\mathbf{u}(t)-\mathbf{u}_{\infty}]^\dagger\frac{\d[\mathbf{u}(t)-\mathbf{u}_{\infty}]}{\d t}\right).
    \end{aligned}
\end{equation}
By substituting Eq.~(\ref{equ:error:3}) into Eq.~(\ref{equ:error:4}), one has
\begin{equation*}
    \begin{aligned}
        \frac{\d\Vert \mathbf{u}(t)-\mathbf{u}_{\infty}\Vert_2^2}{\d t}&=-2[\mathbf{u}(t)-\mathbf{u}_{\infty}]^\dagger\mathbf{H}_1[\mathbf{u}(t)-\mathbf{u}_{\infty}]\\
        &\le -2\lambda_{\min}(\mathbf{H}_1)\Vert \mathbf{u}(t)-\mathbf{u}_{\infty}\Vert_2^2.
    \end{aligned}
\end{equation*}
Applying Gr\"onwall's inequality, we obtain
\begin{equation*}
    \begin{aligned}
        \Vert \mathbf{u}(t)-\mathbf{u}_{\infty}\Vert_2^2\le e^{-2\lambda_{\min}(\mathbf{H}_1)t}\Vert \mathbf{u}_0-\mathbf{u}_{\infty}\Vert_2^2.
    \end{aligned}
\end{equation*}
This completes the proof.
\end{proof}
\end{lemma}
One can further show that the error $\Vert \mathbf{u}_{\text{schr}}(t)-(\mathbf{u}_{\text{schr}})_{\infty}\Vert_2$ of the Schr\"odingerized system in Eq.~(\ref{equ:system:6}) can be controlled by $\Vert \mathbf{u}(t)-\mathbf{u}_{\infty}\Vert_2$. Here, we present the result without proof.
\begin{lemma}
\label{lemma:error:4}
\par The 2-norms of $\mathbf{u}_{\text{schr}}(t)-(\mathbf{u}_{\text{schr}})_{\infty}$ and $\mathbf{u}(t)-\mathbf{u}_{\infty}$ satisfy
\begin{equation*}
    \begin{aligned}
        \Vert\mathbf{u}_{\text{schr}}(t)-(\mathbf{u}_{\text{schr}})_{\infty}\Vert_2\le\Vert\mathbf{u}(t)-\mathbf{u}_{\infty}\Vert_2.
    \end{aligned}
\end{equation*}
\qed
\end{lemma}

\subsection{Query Complexity Analysis}
\par In quantum computing, complexity is commonly measured by the number of queries. Berry \textit{et al.} \cite{Berry2015Hamiltonian} gave the following estimate for the query complexity of Hamiltonian simulation.
\begin{lemma}
\label{lemma:berry}
\par \cite{Berry2015Hamiltonian} An $s$-sparse Hamiltonian $H$ acting on $m_H$ qubits can be simulated with error at most $\delta$ using
\begin{equation}
    \begin{aligned}
        \mathcal{Q}(H)=\mathcal{O}\left(\chi\frac{\log(\chi/\delta)}{\log\log(\chi/\delta)}\right)
    \end{aligned}
\end{equation} 
queries and 
\begin{equation}
    \begin{aligned}
        \mathcal{C}(H)=\mathcal{O}\left(\chi[m_H+\log^{2.5}(\chi/\delta)]\frac{\log(\chi/\delta)}{\log\log(\chi/\delta)}\right)
    \end{aligned}
\end{equation}
additional two-qubit gates, where $\chi=s\Vert H\Vert_{\max}T$ and $T$ is the evolution time.
\qed
\end{lemma}
\par The query complexity $\mathcal{Q}$ is determined by the sparsity $s$, the matrix max norm $\Vert \mathbf{H}_{\text{schr}}\Vert_{\max}$ of $\mathbf{H}_{\text{schr}}$ (defined as the maximum modulus of its entries), and the evolution time $T_{\text{evol}}$. However, unlike existing Schr\"odingerization-based methods, the evolution time $T_{\text{evol}}$ in our quantum IMEX scheme is not numerically equal to the target time $T$ in Eq.~(\ref{equ:structure:2}).
\par We therefore estimate the evolution time $T_{\text{evol}}$ required by the quantum IMEX scheme under an error tolerance $\delta$. Compared with previous approaches, this estimate is more involved, so we provide the details below.

\subsubsection{Estimate of the Evolution Time\texorpdfstring{ $T_{\text{evol}}$}{}}

\par As mentioned earlier, we may use the time at which the steady-state error falls below $\delta$ as an estimate of the evolution time $T_{\text{evol}}$. Here, we use Lemma \ref{lemma:error:2} to derive a sufficient lower bound on the required evolution time, although it may not be optimal. First, we compute $\mathbf{H}_1=\frac{\mathbf{H}+\mathbf{H}^\dagger}{2}$ as follows:
\begin{equation}
    \begin{aligned}
        \label{equ:time:1}
        \mathbf{H}_1=
        \begin{bmatrix}
        \frac{P_{N_t-1}+P_{N_t-1}^\dagger}{2} & -\frac{Q_{N_t-1}}{2} &&\\
        -\frac{Q_{N_t-1}^\dagger}{2} & \frac{P_{N_t-2}+P_{N_t-2}^\dagger}{2} & -\frac{Q_{N_t-2}}{2} &\\
        &\ddots & \ddots & \ddots \\
        &&& \frac{P_1+P_1^\dagger}{2} & -\frac{Q_1}{2} \\
        &&& -\frac{Q_1^\dagger}{2} & \frac{P_0+P_0^\dagger}{2}\\
        \end{bmatrix}.
    \end{aligned}
\end{equation}
Therefore, using the Weyl inequality shown in Lemma \ref{lemma:appendix:B:2} and combining it with Lemma \ref{lemma:appendix:A:1}, we obtain the following lower-bound estimate for $\lambda_{\min}(\mathbf{H}_1)$:
\begin{equation}
    \begin{aligned}
        \label{equ:time:2}
        \lambda_{\min}(\mathbf{H}_1)\ge \min\limits_{j=0}^{N_t-1}\lambda_{\min}\left(\frac{P_j+P_j^\dagger}{2}\right)-\max\limits_{j=0}^{N_t-1} \Vert Q_j \Vert_2.
    \end{aligned}
\end{equation}
Furthermore, by using Lemma \ref{lemma:error:2}, we obtain the following estimate for the evolution time $T_{\text{evol}}$:
\begin{equation}
    \begin{aligned}
        \label{equ:time:3}
        T_{\text{evol}}\ge \frac{\log\delta^{-1}}{\min\limits_{j=0}^{N_t-1}\lambda_{\min}\left(\frac{P_j+P_j^\dagger}{2}\right)-\max\limits_{j=0}^{N_t-1} \Vert Q_j \Vert_2},
    \end{aligned}
\end{equation}
which requires the condition $\min\limits_{j=0}^{N_t-1} \lambda_{\min}\left(\frac{P_j + P_j^\dagger}{2} \right)>\max\limits_{j=0}^{N_t-1}\Vert Q_j \Vert_2$, and simplifies to $\sup\limits_{t\in [0,T]} \lambda_{\max}\left(\frac{L_1(t)+L_1(t)^\dagger}{2} \right) < 0$ as $\varepsilon \to 0$. For a given PDE, choosing $P_n$ and $Q_n$ so that the denominator is positive yields the corresponding estimate of the evolution time.

\subsubsection{Estimate of the Query Complexity}

\par Based on the above results, we estimate the query complexity of the quantum IMEX schemes. First, we establish the query complexity required to implement the corresponding quantum circuits.
\begin{lemma}
\label{lemma:main:1}
\par For the iterative problem $P_n u_{n+1}=Q_n u_n+b_n,\, n=0,1,\cdots$, the corresponding Hamiltonian system in Eq.~(\ref{equ:system:6}) can be simulated with error at most $\delta$ using
\begin{equation*}
    \begin{aligned}
        \mathcal{Q}_{query}=\mathcal{O}(s\|\mathbf{H}_{\text{schr}}\|_{\max}T_{\text{evol}}\log(s\|\mathbf{H}_{\text{schr}}\|_{\max}T_{\text{evol}})),
    \end{aligned}
\end{equation*}
in which the sparsity term $s$ is given by
\begin{equation*}
\begin{aligned}
s=\max\limits_{n=0}^{N_t-1}\left[s(P_n)+s(Q_n)\right]=\mathcal{O}(s(L)),
\end{aligned}
\end{equation*}
where $s(L)=\sup\limits_{t\in[0,T]}s(L(t))$. The maximum value term $\|\mathbf{H}_{\text{schr}}\|_{\max}$ is estimated as
\begin{equation*}
    \begin{aligned}
        \|\mathbf{H}_{\text{schr}}\|_{\max}&=\mathcal{O}\left(N_p \max\limits_{n=0}^{N_t-1}\left[2\left\Vert P_n\right\Vert_{\max}+\left\Vert Q_n\right\Vert_{\max}+\left\Vert b_n+\mathbbm{1}_{\{n=0\}}Q_0u_0\right\Vert_{\max}\right]\right)\\
        &=\mathcal{O}\left(\log\delta^{-1}(\|L_1\|_{\max}+\|b_1\|_{\max})\right),
    \end{aligned}
\end{equation*}
with $\mathbbm{1}_{\{n=0\}}$ denoting the indicator of the event $n=0$, where the second line records the corresponding simplified scaling in the multiscale regime considered here. Choosing $T_{\text{evol}}$ at the threshold prescribed by Eq.~(\ref{equ:time:3}), one may take
\begin{equation*}
\begin{aligned}
T_{\text{evol}}&=\mathcal{O}\left(\frac{\log\delta^{-1}N_t}{N_t\varepsilon- T\sup\limits_{t\in[0,T]}\lambda_{\max}\left(\frac{L_1(t)+L_1(t)^\dagger}{2}\right)-N_t\varepsilon\sup\limits_{t\in[0,T]}\left\Vert I+\tau \frac{L_2(t) +L_2(t)^\dagger}{2}\right\Vert_2}\right)\\
&=\mathcal{O}\left(\frac{\log\delta^{-1}N_t}{-T\sup\limits_{t\in[0,T]}\lambda_{\max}\left(\frac{L_1(t)+L_1(t)^\dagger}{2}\right)}\right),
\end{aligned}
\end{equation*}
where it is necessary that $\frac{T}{N_t}\sup\limits_{t\in[0,T]}\lambda_{\max}\left(\frac{L_1(t)+L_1(t)^\dagger}{2}\right)+\varepsilon\sup\limits_{t\in[0,T]}\left\Vert I+\tau L_2(t)\right\Vert_2< \varepsilon$, and if $\varepsilon \to 0$, the condition simplifies to $\sup\limits_{t\in[0,T]} \lambda_{\max}\left(\frac{L_1(t)+L_1(t)^\dagger}{2} \right)<0$.
\begin{proof}
\par According to Lemma \ref{lemma:berry}, the relevant quantity in the query-complexity estimate is $\chi=s(\mathbf{H}_{\text{schr}})\Vert \mathbf{H}_{\text{schr}}\Vert_{\max} T_{\text{evol}}$. Below, we compute these three values separately:
\begin{itemize}
\item It can be readily verified that
\begin{equation*}
    \begin{aligned}
        s(\mathbf{H}_{\text{schr}})=s:=\max\limits_{n=0}^{N_t-1}\left[s(P_n)+s(Q_n)\right]=\mathcal{O}(s(L)).
    \end{aligned}
\end{equation*}
\item Given that $\Vert\mathbf{H}_{\text{schr}}\Vert_{\max}\le\Vert\mathbf{H}_{\text{homo},1}\Vert_{\max}\Vert D_\mu\Vert_{\max}+\Vert\mathbf{H}_{\text{homo},2}\Vert_{\max}$, we can prove
\begin{equation*}
    \begin{aligned}
        \Vert \mathbf{H}_{\text{homo},1}\Vert_{\max}&=\mathcal{O}\left(\left\Vert\mathbf{H}+\mathbf{H}^\dagger\right\Vert_{\max}+\left\Vert\mathbf{F}\right\Vert_{\max}\right)\\
        &=\mathcal{O}\left(\max\limits_{n=0}^{N_t-1}\left[2\left\Vert P_n\right\Vert_{\max}+\left\Vert Q_n\right\Vert_{\max}+\left\Vert b_n+\mathbbm{1}_{\{n=0\}} Q_0u_0\right\Vert_{\max}\right]\right),
    \end{aligned}
\end{equation*}
and the same order estimate holds for $\Vert \mathbf{H}_{\text{homo},2}\Vert_{\max}$. Moreover, $\Vert D_{\mu}\Vert_{\max}=N_p=\mathcal{O}(\delta^{-1})$. If one adopts the optimal smooth initialization of \cite{Jin2025Optimal}, then the largest Fourier-mode cutoff satisfies $\Vert D_{\mu}\Vert_{\max}=\mathcal{O}(\log\delta^{-1})$ in the present normalization $\mu_\ell=\pi \ell$, equivalently $N_p=\mathcal{O}(\log\delta^{-1})$. Therefore, it can be concluded that
\begin{equation*}
    \begin{aligned}
        \Vert\mathbf{H}_{\text{schr}}\Vert_{\max}:&=\mathcal{O}\left(N_p\max\limits_{n=0}^{N_t-1}\left[2\left\Vert P_n\right\Vert_{\max}+\left\Vert Q_n\right\Vert_{\max}+\left\Vert b_n+\mathbbm{1}_{\{n=0\}} Q_0u_0\right\Vert_{\max}\right]\right)\\
        &=\mathcal{O}(\log\delta^{-1}(\|L_1\|_{\max}+\|b_1\|_{\max})).
    \end{aligned}
\end{equation*}
\item Eq.~(\ref{equ:time:3}) shows that a sufficient choice of $T_{\text{evol}}$ is
\begin{equation*}
    \begin{aligned}
        T_{\text{evol}}&=\mathcal{O}\left(\log\delta^{-1}\Big/\left(\min\limits_{j=0}^{N_t-1}\lambda_{\min}\left(\frac{P_j+P_j^\dagger}{2}\right)-\max\limits_{j=0}^{N_t-1} \Vert Q_j \Vert_2\right)\right)\\
        &=\mathcal{O}\left(\log\delta^{-1}N_t\Big/\left(N_t\varepsilon-T\sup\limits_{t\in[0,T]}\lambda_{\max}\left(\frac{L_1(t)+L_1(t)^\dagger}{2}\right)-N_t\varepsilon\sup\limits_{t\in[0,T]}\left\Vert I+\tau \frac{L_2(t) +L_2(t)^\dagger}{2}\right\Vert_2\right)\right)\\
        &=\mathcal{O}\left(\log\delta^{-1}N_t\Big/\left(-T\sup\limits_{t\in[0,T]}\lambda_{\max}\left(\frac{L_1(t)+L_1(t)^\dagger}{2}\right)\right)\right).
    \end{aligned}
\end{equation*}
\end{itemize}
Thus, it can be concluded that
\begin{equation*}
    \begin{aligned}
        \chi=\mathcal{O}(s\|\mathbf{H}_{\text{schr}}\|_{\max}T_{\text{evol}}).
    \end{aligned}
\end{equation*}
\par By substituting the order of magnitude of the obtained parameters into Lemma \ref{lemma:berry}, we obtain
\begin{equation*}
    \begin{aligned}
        \mathcal{Q}_{query}=\mathcal{O}(s\|\mathbf{H}_{\text{schr}}\|_{\max}T_{\text{evol}}\log(s\|\mathbf{H}_{\text{schr}}\|_{\max}T_{\text{evol}})).
    \end{aligned}
\end{equation*}
This analysis shows that the query complexity becomes independent of $\varepsilon$ because the evolution time $T_{\text{evol}}$ is decoupled from the physical solution time $T$.
\end{proof}
\end{lemma}
\par During measurement, repeated runs are required to obtain a stable success probability. Here we follow the framework in \cite{Jin2025LinearNonUnitary}, which yields the following estimate.
\begin{lemma}
\label{lemma:success}
\par To simulate the ODE with inhomogeneous term $\frac{\d w(t)}{\d t}=Aw(t)+b$ with an initial condition $w(0)=w_0$ over the time interval $[0,T]$, one can transform it into a homogeneous system by introducing $w_{\text{homo}}(t)=[w(t);\mathbf{1}]$, and the equivalent homogeneous ODE is
\begin{equation*}
    \begin{aligned}
        \frac{\d w_{\text{homo}}(t)}{\d t}=\begin{bmatrix}
            A & \text{diag}(b)\\
            O & O
        \end{bmatrix}w_{\text{homo}}(t).
    \end{aligned}
\end{equation*}
Within this framework, the overall probability of successfully retrieving $w(t)$ is approximately
\begin{equation}
    \begin{aligned}
        \text{Pr}(w)=\frac{1}{2}e^{-2p^{\diamond}}\frac{\|w(T)\|^2}{\|w_0\|^2+T^2\|b\|^2},
    \end{aligned}
\end{equation}
and via amplitude amplification the required repetition count for measurements can be estimated as
\begin{equation}
    \begin{aligned}
        g=\mathcal{O}\left(e^{p^{\diamond}}\frac{\|w_0\|+T\|b\|}{\|w(T)\|}\right),
    \end{aligned}
\end{equation}
where $p^{\diamond} = T\max\{0, \lambda_{\max}(A_1)\}$, with $A_1=\frac{A+A^\dagger}{2}$.
\end{lemma}
\noindent Note that in the quantum IMEX schemes, we evolve the equivalent ODE in Eq.~(\ref{equ:system:3}). If we wish to recover the full discrete solution represented by Eq.~(\ref{equ:system:1}) over the entire time interval $[0,T]$, no additional state selection from $\mathbf{u}(t)$ is required. However, if only the state at the final time $T$ is needed, we must select the first block corresponding to time $T$, and the associated selection probability is approximately $\frac{1}{N_t}$. Based on this observation, we can apply Lemma \ref{lemma:success} directly and obtain the following estimate for the repetition number in simulating Eq.~(\ref{equ:system:1}):
\begin{lemma}
\label{lemma:main:2}
\par Set the initial condition for Eq.~(\ref{equ:system:3}) to be $\mathbf{u}_0=0$, and assume furthermore that the homogeneous extension entering Lemma \ref{lemma:success} is normalized so that the relevant physical branch satisfies $p^{\diamond}=0$. To recover the full discrete state over the time interval $[0,T]$, the repetition number required for simulating Eq.~(\ref{equ:system:1}) is given by:
\begin{equation}
    \begin{aligned}
        \label{equ:repeat:1}
        g=\mathcal{O}\left(\frac{T T_{\text{evol}}\|b_1\|_{2,\max}}{N_t\|u\|_{2,\min}}\right)=\mathcal{O}\left(\frac{\log\delta^{-1}\|b_1\|_{2,\max}}{-\sup\limits_{t\in[0,T]}\lambda_{\max}\left(\frac{L_1(t)+L_1(t)^\dagger}{2}\right)\|u\|_{2,\min}}\right),
    \end{aligned}
\end{equation}
where $\|u\|_{2,\min}=\inf\limits_{t\in[0,T]}\|u(t)\|_2$ and $\|b_1\|_{2,\max}=\sup\limits_{t\in[0,T]}\|b_1(t)\|_2$. Alternatively, if only the state at the final time $T$ is needed, the required repetition number becomes:
\begin{equation}
    \begin{aligned}
        \label{equ:repeat:2}
        g=\mathcal{O}\left(\frac{T T_{\text{evol}}\|b_1\|_{2,\max}}{\sqrt{N_t}\|u\|_{2,\min}}\right)=\mathcal{O}\left(\frac{\log\delta^{-1}\sqrt{N_t}\|b_1\|_{2,\max}}{-\sup\limits_{t\in[0,T]}\lambda_{\max}\left(\frac{L_1(t)+L_1(t)^\dagger}{2}\right)\|u\|_{2,\min}}\right).
    \end{aligned}
\end{equation}
\begin{proof}
\par First, we consider the case in which we directly examine the success probability of the evolution of Eq.~(\ref{equ:system:3}). Here, we substitute $w(t)=\mathbf{u}(t)$, $w_0=\mathbf{u}_0 = 0$, $T=T_{\text{evol}}$, and $b=\mathbf{F}$. As a result, the probability of successfully obtaining $\mathbf{u}(t)$ is given by
\begin{equation*}
    \begin{aligned}
        \text{Pr}(w)
        &\gtrsim \frac{1}{2}\frac{\|\mathbf{u}(T_{\text{evol}})\|_2^2}{T_{\text{evol}}^2\|\mathbf{F}\|_2^2}
        \gtrsim \frac{1}{2}\frac{\sum\limits_{n=0}^{N_t}\|u_n\|_2^2}{T_{\text{evol}}^2\sum\limits_{n=0}^{N_t-1}\|b_n\|_2^2}
        \gtrsim \frac{1}{2}\frac{\|u\|_{2,\min}^2}{T_{\text{evol}}^2\tau^2\|b_1\|_{2,\max}^2},
    \end{aligned}
\end{equation*}
which allows us to derive the repetition numbers as
\begin{equation}
    \begin{aligned}
        g=\mathcal{O}\left(\frac{T T_{\text{evol}}\|b_1\|_{2,\max}}{N_t\|u\|_{2,\min}}\right)=\mathcal{O}\left(\frac{\log\delta^{-1}\|b_1\|_{2,\max}}{-\sup\limits_{t\in[0,T]}\lambda_{\max}\left(\frac{L_1(t)+L_1(t)^\dagger}{2}\right)\|u\|_{2,\min}}\right),
    \end{aligned}
\end{equation}
where the displayed scaling corresponds to the normalization regime $p^{\diamond}=0$ assumed in the statement. In the case where only the state at time $T$ is required, the success probability is further reduced by a factor of $\frac{1}{N_t}$. This increases the repetition number to
\begin{equation}
    \begin{aligned}
        g=\mathcal{O}\left(\frac{T T_{\text{evol}}\|b_1\|_{2,\max}}{\sqrt{N_t}\|u\|_{2,\min}}\right)=\mathcal{O}\left(\frac{\log\delta^{-1}\sqrt{N_t}\|b_1\|_{2,\max}}{-\sup\limits_{t\in[0,T]}\lambda_{\max}\left(\frac{L_1(t)+L_1(t)^\dagger}{2}\right)\|u\|_{2,\min}}\right).
    \end{aligned}
\end{equation}
This completes the proof.
\end{proof}
\end{lemma}
\par Under the dissipativity condition $\sup\limits_{t\in[0,T]}\lambda_{\max}\left(\frac{L_1(t)+L_1(t)^\dagger}{2}\right)<0$ and under the same normalization regime $p^{\diamond}=0$ from Lemma \ref{lemma:main:2}, the following two lemmas yield a quantum algorithm for the multiscale problem through the iterative system in Eq.~(\ref{equ:system:1}) whose query complexity is \textit{independent} of the scaling parameter $\varepsilon$.
\begin{theorem}
\label{theorem:main}
\par Under the same dissipativity condition and normalization assumption, there exists a quantum algorithm that simulates the original multiscale problem over the time interval $[0,T]$ through Eq.~(\ref{equ:system:1}), with overall query complexity
\begin{equation}
    \begin{aligned}
        \mathcal{Q}=\mathcal{O}\left(\frac{\|b_1\|_{2,\max}}{\|u\|_{2,\min}}\frac{s(\|L_1\|_{\max}+\|b_1\|_{\max})N_t(\log\delta^{-1})^3}{\left(-\sup\limits_{t\in[0,T]}\lambda_{\max}\left(\frac{L_1(t)+L_1(t)^\dagger}{2}\right)\right)^2}\log\left(s(\|L_1\|_{\max}+\|b_1\|_{\max})\frac{N_t(\log\delta^{-1})^2}{-\sup\limits_{t\in[0,T]}\lambda_{\max}\left(\frac{L_1(t)+L_1(t)^\dagger}{2}\right)}\right)\right),
    \end{aligned}
\end{equation}
which is independent of $\varepsilon$, where the detailed definitions are provided in Lemmas \ref{lemma:main:1} and \ref{lemma:main:2}. Furthermore, assume that all entries of $L_1$, $L_2$, $b_1$, and $b_2$ are $\mathcal{O}(1)$, and that all components of $u$ and its first two time derivatives are also $\mathcal{O}(1)$. Choosing $N_t$ at the accuracy threshold prescribed by Eq.~(\ref{equ:euler:7}), we obtain the following informal estimate of the overall query complexity:
\begin{equation}
    \begin{aligned}
        \mathcal{Q}=\mathcal{O}\left(\frac{T\delta^{-1}(\log\delta^{-1})^3}{\left(-\sup\limits_{t\in[0,T]}\lambda_{\max}\left(\frac{L_1(t)+L_1(t)^\dagger}{2}\right)\right)^3}\right),
    \end{aligned}
\end{equation}
\end{theorem}
\par As the detailed quantum implementation is based on the same Schr\"odingerization framework as described in \cite{Jin2025LinearNonUnitary}, we omit a full exposition here and refer the reader to this reference for technical details.
\section{Improved Quantum IMEX Schemes for More General \texorpdfstring{$\mathbf{H}_1$}{H1}}

\label{section:timein}
\par The framework proposed above imposes relatively strict requirements on the Hermitian part $\mathbf{H}_1$. To generalize this framework, we use a refined error-analysis approach. In this section, we examine the special case of time-independent parameters, with $P_n = P$, $Q_n = Q$, and $b_n=b$. For Eq. \eqref{equ:system:1}, this corresponds to the following matrices $\mathbf{H}$ and $\mathbf{F}$:
\begin{equation*}
    \begin{aligned}
        \mathbf{H}=
        \begin{bmatrix}
        P & -Q &&\\
        & P & -Q &\\
        && \ddots & \ddots \\
        &&& P & -Q \\
        &&&& P\\
        \end{bmatrix},\quad
        \mathbf{F}=
        \begin{bmatrix}
        b\\
        b\\
        \vdots\\
        b\\
        Qu_0+b
        \end{bmatrix},
    \end{aligned}
\end{equation*}
where we do not require that $P$ and $Q$ commute. For brevity, this section focuses on estimating the evolution time $T_{\text{evol}}$, while omitting established procedures such as the associated query-complexity analysis, which has already been detailed in Section \ref{section:timede}. We first present an error-estimation method based on the matrix exponential, which provides an alternative estimate of the evolution time that can be tighter in some settings.
\subsection{Methods for Estimating Global Error Based on Matrix Exponential}
\par In Section \ref{section:timede}, we introduced two lemmas for estimating the evolution time $T_{\text{evol}}$. However, the methods in Lemmas \ref{lemma:error:1} and \ref{lemma:error:2} still impose requirements on the matrix eigenvalues. Their relative strengths and limitations can be summarized by the following chain of inequalities \cite{Krovi2023Improved}.
\begin{lemma}
\par For any matrix $A$, there is
\begin{equation*}
    \begin{aligned}
        \exp(\alpha(A) t)\le \|\exp(At)\|_2\le \exp(\mu(A)t)\le \exp(\|A\|_2t),
    \end{aligned}
\end{equation*}
where $\alpha(A)$ is the maximum real part of the eigenvalues of matrix $A$, and $\mu(A)=\lambda_{\max}\left(\frac{A+A^\dagger}{2}\right)$ is the logarithmic norm of $A$.
\qed
\end{lemma}
\noindent Accordingly, we use the following lemma for error estimation.
\begin{lemma}
\label{lemma:error:3}
\par For the dynamical system presented in Eq.~(\ref{equ:system:3}), $\mathbf{u}(t)$ converges to the steady state $\mathbf{u}_{\infty}$ in the sense
\begin{equation}
    \begin{aligned}
        \label{equ:error:5}
        \Vert \mathbf{u}(t)-\mathbf{u}_{\infty}\Vert_2\le \Vert e^{-\mathbf{H}t}\Vert_2\Vert \mathbf{u}_0-\mathbf{u}_{\infty}\Vert_2.
    \end{aligned}
\end{equation}
\begin{proof}
\par The analytical solution to this ODE given in Eq.~(\ref{equ:error:3}) is
\begin{equation*}
    \begin{aligned}
        \mathbf{u}(t)-\mathbf{u}_{\infty}=e^{-\mathbf{H}t}[\mathbf{u}_0-\mathbf{u}_{\infty}].
    \end{aligned}
\end{equation*}
Therefore, by taking 2-norm on both sides and utilizing the consistency property of norms, we obtain:
\begin{equation*}
    \begin{aligned}
        \Vert\mathbf{u}(t)-\mathbf{u}_{\infty}\Vert_2\le \Vert e^{-\mathbf{H}t}\Vert_2\Vert \mathbf{u}_0-\mathbf{u}_{\infty}\Vert_2.
    \end{aligned}
\end{equation*}
This completes the proof.
\end{proof}
\end{lemma}
\subsection{Query Complexity Analysis}
\par Direct computation of $e^{-\mathbf{H}t}$ can be cumbersome. We therefore use the Laplace transform and its inverse to obtain an exact expression. The details are provided in Appendix Section \ref{section:B}. We obtain the following upper bound for its norm:
\begin{equation}
    \begin{aligned}
        \label{equ:norm:3new}
        \Vert e^{-\mathbf{H}t}\Vert_2\le \Vert e^{-Pt+\int_0^t \Vert\tilde{Q}(\tau) \Vert_2\d\tau}\Vert_2,
    \end{aligned}
\end{equation}
where $\tilde{Q}(t)=e^{Pt}Qe^{-Pt}$. If $Q=I$, i.e., for the fully implicit scheme with $P=I-\tau L$, one obtains the simpler upper bound
\begin{equation}
    \begin{aligned}
        \label{equ:norm:4new}
        \Vert e^{-\mathbf{H}t}\Vert_2\le e^t\Vert e^{-Pt}\Vert_2=\Vert e^{(-P+I)t}\Vert_2=\Vert e^{\tau Lt}\Vert_2.
    \end{aligned}
\end{equation}
This suggests that the IMEX method may admit further improvements in the time-independent case, and here we outline a basic route in that direction.
\par Although Eq.~(\ref{equ:norm:4new}) provides a sharper bound, it does not yield an explicit closed-form expression for $T_{\text{evol}}$. Following Krovi \cite{Krovi2023Improved}, we present the following three approximate estimates. First, the following logarithmic-norm lemma can be regarded as a special case of Eq.~(\ref{equ:time:3}):
\begin{lemma}
\label{lemma:lognorm}
\par Let $A$ be an arbitrary square matrix. The induced matrix norm of its exponential satisfies the following upper bound, valid for any consistent matrix norm:
\begin{equation}
    \begin{aligned}
        \Vert  e^{At}\Vert_2 \le e^{\mu(A)t},
    \end{aligned}
\end{equation}
in which $\mu(A)$ is
\begin{equation*}
    \begin{aligned}
        \mu(A)=\lim\limits_{h\to 0}\frac{\Vert  I+hA\Vert_2 -1}{h}=\lambda_{\max}\left(\frac{A+A^\dagger}{2}\right).
    \end{aligned}
\end{equation*}
\begin{proof}
\par Since $h>0$, we can set $n=\frac{t}{h}$, and then use the compatibility of the matrix exponential norm to obtain
\begin{equation*}
    \begin{aligned}
        \Vert e^{At}\Vert_2 \le \Vert e^{hA}\Vert_2^n=\Vert e^{I+hA}\Vert_2^n\cdot e^{-n}\le e^{(\Vert I+hA\Vert_2 -1)n}=e^{\frac{\Vert  I+hA\Vert_2 -1}{h}t},
    \end{aligned}
\end{equation*}
and
\begin{equation*}
    \begin{aligned}
        \Vert e^{At}\Vert_2 \le \lim\limits_{h\to0}e^{\frac{\Vert  I+hA\Vert_2 -1}{h}t}=e^{\mu(A)t}.
    \end{aligned}
\end{equation*}
The exact value of $\mu(A)$ is proved in Lemma \ref{lemma:appendix:D:1}, and this completes the proof.
\end{proof}
\end{lemma}
\noindent For the following two lemmas on the Jordan and Schur decompositions, we directly cite the conclusions from Krovi's work \cite{Krovi2023Improved} without proof.

\begin{lemma}
\par \cite{Krovi2023Improved} Let $A$ be an arbitrary square matrix with the largest eigenvalue being $\lambda$, and let its Jordan decomposition be $A = V^{-1} J V$, $\alpha$ be the size of its largest Jordan block. Then the following inequality holds:
\begin{equation}
    \begin{aligned}
        \Vert  e^{At}\Vert_2 \le \kappa(V)\alpha\max\limits_{0\le r\le\alpha-1}\frac{t^r}{r!}e^{\lambda t},
    \end{aligned}
\end{equation}
where $\kappa(V)$ is the condition number of $V$.
\end{lemma}
\begin{lemma}
\par \cite{Krovi2023Improved} Let $A$ be an arbitrary square matrix with the largest eigenvalue being $\lambda$, and let its Schur decomposition be $A = U(D + N)U^\dagger$, where $D$ is a diagonal matrix and $N$ is a strictly upper triangular matrix. Then the following inequality holds:
\begin{equation}
    \begin{aligned}
        \Vert  e^{At}\Vert_2 \le \sum\limits_{k=0}^n\frac{(\Vert N\Vert_2 t)^k}{k!}e^{\lambda t}.
    \end{aligned}
\end{equation}
Since the Schur decomposition of $A$ is not unique, we can choose the one with the smallest $\Vert N\Vert $.
\end{lemma}
\begin{remark}
\par Here we provide an example to illustrate a potential advantage of the method in this section over the result of Eq.~(\ref{equ:time:3}). Consider $A = \begin{bmatrix} -1 & a \\ 0 & -2 \end{bmatrix}$ with $a>0$.
\begin{itemize}
    \item The eigendecomposition of $A = V^{-1} \Lambda V$ is given by $\Lambda = \begin{bmatrix} -1 & 0 \\ 0 & -2 \end{bmatrix}$, $V = \begin{bmatrix} 1 & a \\ 0 & -1 \end{bmatrix}$, which corresponds to a Jordan decomposition with the largest Jordan block of size 1. Thus, we have $\kappa(V) \le (1 + a)^2$, $\lambda = -1$.
    \item $\lambda_{\max}\left( \frac{A + A^\dagger}{2} \right) = \frac{-3 + \sqrt{1 + a^2}}{2}$. When $a$ is sufficiently large, one has $\kappa(V) e^{\lambda t} < e^{\mu(A) t}$.
\end{itemize}
\end{remark}

\section{Applications to Specific PDEs}
\subsection{The Linear Heat Equation with Time-Dependent Coefficients}
\par Let us consider the following linear heat equation with time-dependent diffusion coefficients, together with boundary and initial conditions:
\begin{gather}
    \label{equ:heat}
    \nonumber\partial_t u(x,t)=\sum\limits_{i=1}^d\frac{a_i(t)}{\varepsilon}\partial_{x_i x_i}u(x,t),\text{ in $\Omega:(0,1)^d$, $0<t<T$,}\\
    u(x,0)=u_0(x),\quad u(\cdot,t)=0\text{ on $\partial\Omega$},
    \label{equ:heat:1}
\end{gather}
in which $u(x,t)$ represents the temperature distribution (or another diffusive quantity such as concentration) at position $x$ and time $t$, and $a_i(t)/\varepsilon>0$ is the thermal diffusivity (or diffusion coefficient) associated with the $i$-th spatial variable at time $t$. We assume that $a_i(t)$ is of order $\mathcal{O}(1)$, while allowing it to vary over time or take very large values. Several quantum algorithms for the heat equation are already available. Jin \textit{et al.} \cite{Jin2022TimeCA} combined finite differences with an HHL-based algorithm to obtain a QLSA algorithm for the heat equation, and estimated its query complexity as $\mathcal{O}(N_x^2\log (N_x/d))$, where $N_x$ is the number of discrete spatial variables, i.e., $N_x \sim h^{-1}$. Subsequently, Jin \textit{et al.} \cite{Jin2024Schrodingerization} proposed a Hamiltonian-simulation algorithm for the heat equation. They used the Schr\"odingerization method to transform the discretized spatial matrix into a Hermitian matrix. Ignoring the additional dependence on the Schr\"odingerization parameters, their time complexity is $\mathcal{O}(d\varepsilon^{-1}N_x^2\log N_x)$ (where $\mathcal{O}$ ignores the $\log\log$ term), which is comparable to that of the QLSA-based algorithm.
\subsubsection{Finite Difference Schemes}
\par We discretize this $d$-dimensional equation directly in both time and space. We divide the time interval $t\in[0,T]$ into $N_t+1$ points, where $0=t_0<t_1<\cdots<t_{N_t-1}<t_{N_t}=T$, and $\tau=T/N_t$. For each spatial dimension $x_k\in[0,1]$, we divide it into $N_x+1$ points, where $0=x_{k,0}<x_{k,1}<\cdots<x_{k,N_x-1}<x_{k,N_x}=1$, and $h=1/N_x$. 
\par At time $t=n\tau$, let $u_n=\text{vec}[U_n]$, where $U_n$ is a $d$-dimensional tensor whose $k$-th index corresponds to the $k$-th spatial dimension. We also denote $(a_k)_n=a_k(n\tau)$. We discretize the temporal derivative as $\partial_t u=(u_{n+1}-u_n)/\tau$, and the Laplacian in each spatial dimension as $\partial_{xx} u=h^{-2}L_hu$. Combining these discretizations, we obtain the following scheme:
\begin{equation*}
    \begin{aligned}
        P_n u_{n+1}=Q_n u_n+b_n,
    \end{aligned}
\end{equation*}
in which the specific matrix format is expressed as
\begin{gather*}
    P_n=\varepsilon I^{\otimes d}-\lambda \left(\sum\limits_{k=0}^{d-1}I^{\otimes k}\otimes (a_{k+1})_nL_h\otimes I^{\otimes(d-k-1)}\right):=\varepsilon I-\lambda L_{h,d},\quad 
    Q_n=\varepsilon I,\\
    b_n=\varepsilon[u(n\tau,x_0,\cdots,x_0),0,\cdots,0,u(n\tau,x_{N_x},\cdots,x_0),\cdots,u(n\tau,x_{N_x},\cdots,x_0);0;\cdots;0;u(n\tau,x_{N_x},\cdots,x_{N_x})],\\
    u_0=[u(0,x_0,\cdots,x_0);u(0,x_1,\cdots,x_0);\cdots;u(0,x_{N_x},\cdots,x_0);\cdots;u(0,x_{N_x},\cdots,x_0);\cdots;u(0,x_{N_x},\cdots,x_{N_x})],
\end{gather*}
where $\lambda=\frac{\tau}{h^2}$, $L_h$ is called the second derivative matrix, and its specific format is as follows:
\begin{equation*}
    \begin{aligned}
        L_h=
        \begin{bmatrix}
        -2 & 1 &&&\\
        1 & -2 & 1 &&\\
        & \ddots & \ddots & \ddots &\\
        && 1 & -2 & 1 \\
        &&& 1 & -2
        \end{bmatrix}.
    \end{aligned}
\end{equation*}
Note that this discretization uses a fully implicit scheme, and under our assumptions it is sufficient that all eigenvalues of $P_n$ be positive. Therefore, the system is unconditionally stable: there is no CFL restriction, and the equation can still be solved efficiently in the presence of stiff terms.

\subsubsection{Query Complexity Analysis}

\par Below we analyze the query complexity, which requires computing $\chi = s\|\mathbf{H}_{\text{schr}}\|_{\max}T_{\text{evol}}$. In this calculation, the estimate of the evolution time $T_{\text{evol}}$ is particularly important. Using Lemma \ref{lemma:main:1} and Eq.~(\ref{equ:time:3}), one can see that $T_{\text{evol}}$ depends on $\min\limits_{n=0}^{N_t-1}\lambda_{\min}\left(\frac{P_n+P_n^\dagger}{2}\right)$; for a specific time index $n$, one has
\begin{equation*}
    \begin{aligned}
        \lambda_{\min}\left(\frac{P_n+P_n^\dagger}{2}\right)&=\varepsilon-\lambda \lambda_{\max}(L_{h,d}).
    \end{aligned}
\end{equation*}
Specifically, through Weyl's inequality, we give an upper bound for $\lambda_{\max}(L_{h,d})$:
\begin{equation*}
    \begin{aligned}
        \lambda_{\max}(L_{h,d})&\le \sum\limits_{k=0}^{d-1}\lambda_{\max}\left(I^{\otimes k}\otimes (a_{k+1})_nL_h\otimes I^{\otimes(d-k-1)}\right)\\
        &\le 
        {\|a\|_{\max}}\sum\limits_{k=0}^{d-1}\lambda_{\max}\left(L_h\right)\\
        &=2d\|a\|_{\max}\left(-1+\cos\frac{\pi}{N_x+1}\right).
    \end{aligned}
\end{equation*}
where $\|a\|_{\max} = \max\limits_{k,n}(a_k)_n$.
Therefore, one can obtain the estimate of the evolution time through Eq.~(\ref{equ:time:3}) as
\begin{equation*}
    \begin{aligned}
        T_{\text{evol}}\ge\frac{\log\delta^{-1}}{\min\limits_{n=0}^{N_t-1}\lambda_{\min}\left(\frac{P_n+P_n^\dagger}{2}\right)-\varepsilon}&\gtrsim \log\delta^{-1}\cdot\frac{N_x^2}{\lambda d}.
    \end{aligned}
\end{equation*}
\par 
For the sparsity $s$, the sparsities of $P_n$ and $Q_n$ are of orders $\mathcal{O}(d)$ and $\mathcal{O}(1)$, respectively; hence $s=\mathcal{O}(d)$. Since $\|\mathbf{H}_{\text{schr}}\|_{\max}$ depends on the ratio $\frac{\tau}{h}^2=\mathcal{O}(1)$, we obtain $\|\mathbf{H}_{\text{schr}}\|_{\max}=\mathcal{O}(\log\delta^{-1})$. Thus the Hamiltonian-simulation stage requires $\mathcal{O}(N_x^2\log N_x(\log\delta^{-1})^2)$ queries. For the measurement cost, if each entry of $u_n$ and $u_0$ is $\mathcal{O}(1)$, then $\|\mathbf{F}\|_2=\mathcal{O}(\varepsilon\sqrt{N_x^d})$ and $\|\mathbf{u}(0)\|_2=\|\mathbf{u}(T_{\text{evol}})\|_2=\mathcal{O}(\sqrt{N_tN_x^d})$. Therefore the amplitude-amplification overhead scales as $\mathcal{O}(1)$ for fixed $d$, and the overall query complexity is $\mathcal{O}(N_x^2\log N_x(\log\delta^{-1})^2)$ ($\mathcal{O}(N_x^2(\log N_x)^3)$ if $N_x=\delta^{-1}$). In this homogeneous-boundary case, the measurement overhead is controlled by the initial-data block rather than by the boundary forcing. In addition, under this fully implicit formulation, our approach avoids an explicit CFL-type restriction and remains stable for large coefficients $a_i(t)$, which may be beneficial for stiff multiscale regimes.

\subsubsection{Numerical Example}

\par To illustrate the performance of our method for the linear heat equation, we consider a time-dependent PDE with sufficiently large coefficients $a_i(t)$ that increase over time. 
\par Specifically, we perform numerical simulations in both one-dimensional (1D) and two-dimensional (2D) cases (Figs. \ref{fig:result:1} and \ref{fig:result:2}). In the 1D case, the black curve denotes the numerical solution obtained using the classical scheme for Eq.~(\ref{equ:heat:1}), while the black dots denote the solution obtained by our quantum method in Eq.~(\ref{equ:system:6}). In the 2D case, we plot the three-dimensional surfaces produced by both methods.
\par In both cases, the two numerical solutions agree well, supporting the consistency of the proposed approach.

\begin{figure}[htbp]
\centering{\includegraphics[width=0.375\linewidth]{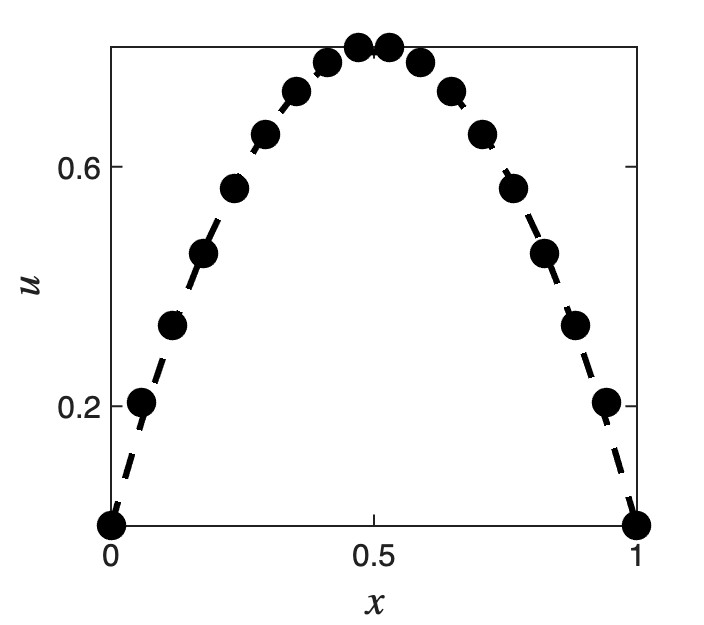}}
    \caption{Numerical results for Eq.~(\ref{equ:heat:1}) in the 1D case after Schr\"odingerization, shown at $t=0.1$, with $a(t)=100/(t+1)$ and $u_0(x)=1$ for $x\in[0,1]$. The spatial and temporal discretization parameters are $\Delta x = 2^{-4}$ and $\Delta t = \Delta x^2/2$, respectively. Results from our scheme are shown by circular markers and compared with the classical-scheme solution shown by the black curve.}
    \label{fig:result:1}
\end{figure}
\begin{figure}[htbp]
    \centering
    \subfigure[The classical scheme]{\includegraphics[width=0.375\linewidth]{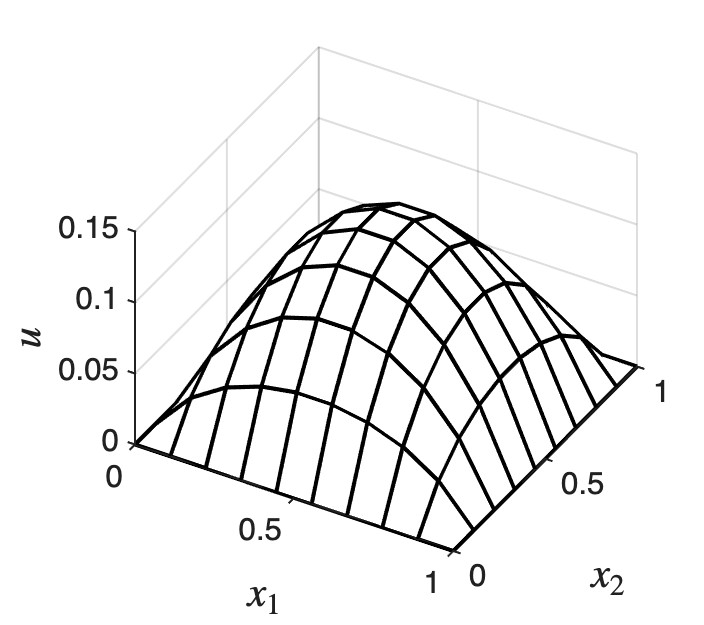}}\qquad
    \subfigure[The quantum IMEX scheme]{\includegraphics[width=0.375\linewidth]{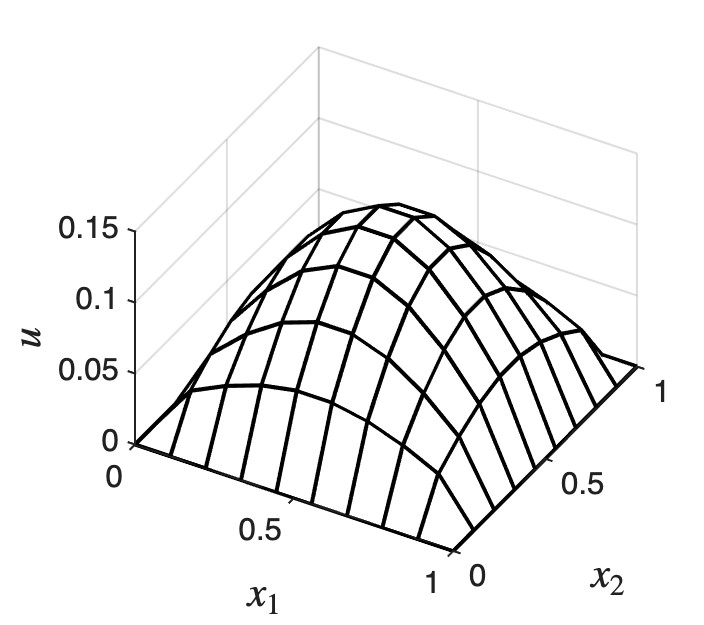}}
    \caption{Numerical results for Eq.~(\ref{equ:heat:1}) in the 2D case after Schr\"odingerization, shown at $t=0.1$. Panel (a) shows the classical scheme, and panel (b) shows the quantum IMEX scheme. The spatial and temporal discretization parameters are $\Delta x_1 = \Delta x_2 = 2^{-3}$ and $\Delta t = \Delta x_1^2/2 = \Delta x_2^2/2$, respectively.}
    \label{fig:result:2}
\end{figure}

\subsection{The Multiscale Telegraph Equation}
\par We consider the following multiscale telegraph equation \cite{Jin2000DiffusiveRelaxation}:
\begin{equation}
    \begin{aligned}
        \label{equ:multiscale:1}
        \partial_tu(t,x)+\partial_xv(t,x)&=0,\\
        \varepsilon^2\partial_tv(t,x)+a(t)\partial_xu(t,x)&=-v(t,x),\quad (0<\varepsilon\ll 1),
    \end{aligned}
\end{equation}
in which $a(t)$ is the propagation speed varying with time, and $\varepsilon$ is the scaling parameter. Because of the numerical stiffness induced by the convection and collision terms \cite{Jin2000DiffusiveRelaxation}, this system is computationally demanding. Jin \textit{et al.} \cite{Jin2000DiffusiveRelaxation} addressed this difficulty by reformulating it as a linear hyperbolic system with a stiff relaxation term, known as the diffusion relaxation system:
\begin{equation}
    \begin{aligned}
        \label{equ:multiscale:2}
        \partial_tu(t,x)+\partial_xv(t,x)&=0,\\
        \partial_tv(t,x)+\partial_xu(t,x)&=-\frac{1}{\varepsilon^2}(v(t,x)+(a(t)-\varepsilon^2)\partial_xu(t,x)),\quad (0<\varepsilon\ll 1).
    \end{aligned}
\end{equation}
\par The construction of quantum algorithms for multiscale equations differs from that for general equations because, if the stiff terms are not handled properly, the query complexity will depend on the scaling parameter $\varepsilon$, which may offset the quantum advantage. For the multiscale telegraph equation, effective HHL-based quantum algorithms have already been studied. Jin \textit{et al.} \cite{Jin2022TimeCA} presented quantum algorithms for the special case $a=1$, based on the IMEX scheme and the diffusive relaxation scheme, respectively. The final query complexity of both results is $\mathcal{O}(N_x^2\log N_x)$. Furthermore, He \textit{et al.} \cite{He2023TimeCA} studied more general linear transport equations and likewise derived a query complexity of $\mathcal{O}(N_x^2\log N_x)$. To the best of our knowledge, Hamiltonian-simulation-based treatments of this setting remain relatively limited. 
Compared with more general frameworks such as \cite{Hu2024QuantumMultiscale}, the present IMEX time discretization is designed to reduce scaling-parameter dependence and facilitate the use of Hamiltonian simulation for this stiffness problem.

\subsubsection{IMEX Asymptotic-Preserving Schemes}
\par The core idea of AP (Asymptotic-Preserving) is to design numerical methods that preserve the asymptotic limit from microscopic models to macroscopic models within a discretized framework \cite{Jin2022AsymptoticpreservingSF}. We consider the IMEX method to solve Eq.~(\ref{equ:multiscale:2}) \cite{Jin2022AsymptoticpreservingSF}:
\begin{equation}
    \begin{aligned}
        \label{equ:multiscale:3}
        &\frac{u_{n+1,j}-u_{n,j}}{\tau}+\frac{v_{n+1,j+1}-v_{n+1,j-1}}{2h}-\frac{h^{1/\beta}}{2}\frac{u_{n,j-1}-2u_{n,j}+u_{n,j+1}}{h^2}=0,\\
        &\frac{v_{n+1,j}-v_{n,j}}{\tau}+\frac{u_{n,j+1}-u_{n,j-1}}{2h}-\frac{h^{1/\beta}}{2}\frac{v_{n,j-1}-2v_{n,j}+v_{n,j+1}}{h^2}\\
        &\qquad\qquad\qquad\qquad\qquad\qquad\qquad\qquad\qquad\qquad\qquad=-\frac{1}{\varepsilon^2}\left(v_{n+1,j}+(a_n-\varepsilon^2)\frac{u_{n+1,j+1}-u_{n+1,j-1}}{2h}\right),
    \end{aligned}
\end{equation}
in which $u_{n,j}$ represents the state at the $j$-th spatial point and the $n$-th time step in the discrete setting, and $\beta\in[1,+\infty)$. As before, we discretize the spatial indices as $j=1,\cdots,N_x-1$, the time steps as $n=0,\cdots,N_t-1$, and denote the spatial and temporal step sizes by $h$ and $\tau$, respectively. We also denote $a_n=a(n\tau)$. The relaxation term is treated implicitly in this scheme. Within our iterative framework, the nodal states can be obtained without explicitly inverting the matrix. Furthermore, letting $u_n=[u_{n,1};\cdots;u_{n,N_x-1}]$ and $v_n=[v_{n,1};\cdots;v_{n,N_x-1}]$, we obtain the matrix form of Eq.~(\ref{equ:multiscale:3}) as follows:
\begin{equation}
    \begin{aligned}
        \label{equ:multiscale:4}
        &\frac{u_{n+1}-u_n}{\tau}+\frac{1}{2h}M_h v_{n+1}-\frac{1}{2h^{2-1/\beta}}L_h u_n-\frac{1}{2h^{2-1/\beta}}(b_{1,n+1}-c_{2,n+1})=0,\\
        &\frac{v_{n+1}-v_n}{\tau}+\frac{1}{2h}M_h u_n-\frac{1}{2h^{2-1/\beta}}L_h v_n+\frac{1}{2h^{2-1/\beta}}(b_{2,n}-c_{1,n})\\
        &\qquad\qquad\qquad\qquad\qquad\qquad\qquad\qquad\qquad\qquad=-\frac{1}{\varepsilon^2}\left(v_{n+1}+\frac{(a_n-\varepsilon^2)}{2h}M_h u_{n+1}+\frac{(a_n-\varepsilon^2)}{2h}b_{2,n+1}\right).
    \end{aligned}
\end{equation}
where the non-homogeneous term is
\begin{equation*}
    \begin{aligned}
        b_{1,n}=[u_{n,0};0;\cdots;0;u_{n,N_x}],\quad 
        b_{2,n}=[-u_{n,0};0;\cdots;0;u_{n,N_x}],\\
        c_{1,n}=[v_{n,0};0;\cdots;0;v_{n,N_x}],\quad
        c_{2,n}=[-v_{n,0};0;\cdots;0;v_{n,N_x}],
    \end{aligned}
\end{equation*}
and $L_h$ is called the second derivative matrix, as previously defined. $M_h$ is referred to as the central difference matrix, and its specific form is as follows
\begin{equation*}
    \begin{aligned}
        M_h =
        \begin{bmatrix}
        0 & 1 &&&\\
        -1 & 0 & \ddots &&\\
        & \ddots & \ddots & \ddots &\\
        && \ddots & 0 & 1 \\
        &&& -1 & 0 \\
        \end{bmatrix}. 
    \end{aligned}
\end{equation*}
\par Eq.~(\ref{equ:multiscale:4}) represents a system of two variables. Let $w_n:=[u_n;v_n]$; then Eq.~(\ref{equ:multiscale:4}) can be written in the following iterative form:
\begin{equation}
    \begin{aligned}
        \label{equ:multiscale:5}
        P_n w_{n+1}=Q_n w_n+b_n,
    \end{aligned}
\end{equation}
in which $\lambda=\frac{\tau}{h}$ and $A_n:=a_n-\varepsilon^2>0$, and
\begin{gather*}
    P_n=\begin{bmatrix}
        I & \frac{\lambda}{2}M_h\\
        \frac{\lambda A_n}{2\varepsilon^2}M_h & \left(1+\frac{\tau}{\varepsilon^2}\right) I
    \end{bmatrix},\quad
    Q_n=\begin{bmatrix}
        I+\frac{\lambda}{2h^{1-1/\beta}}L_h & O\\
        -\frac{\lambda}{2}M_h & I+\frac{\lambda}{2h^{1-1/\beta}}L_h
    \end{bmatrix},\\
    b_n=
    \begin{bmatrix}
        \frac{\lambda}{2h^{1-1/\beta}}(b_{1,n+1}-c_{2,n+1})\\
        -\frac{\lambda}{2h^{1-1/\beta}}(b_{2,n}-c_{1,n})-\frac{\lambda A_n}{2\varepsilon^2}b_{2,n+1}
    \end{bmatrix}.
\end{gather*}
Note that here we use an implicit-explicit hybrid method, so the CFL condition must still be respected, i.e., $\lambda=\frac{\tau}{h}<1$. Under the updated discretization, the singular factor appears in both the lower-left and lower-right blocks of $P_n$. To remove the explicit $\varepsilon^{-2}$ dependence and to make the off-diagonal part of $\hat{P}_n$ purely anti-Hermitian, we introduce the rescaling below and define $\hat{P}_n=S_n\cdot P_n\cdot T_n$, $\hat{Q}_n=S_n\cdot Q_n\cdot T_n$, $\hat{b}_n=S_n\cdot b_n$, and $\hat{w}_n=T_n^{-1}w_n$:
\begin{equation*}
    \begin{aligned}
        S_n=\begin{bmatrix}
            I\\
            & \frac{\varepsilon^2}{\sqrt{\tau A_n}}I
        \end{bmatrix},\quad
        T_n=\begin{bmatrix}
            I\\
            & \sqrt{\frac{A_n}{\tau}}I
        \end{bmatrix}.
    \end{aligned}
\end{equation*}
This leads to the following iterative form:
\begin{equation}
    \begin{aligned}
        \label{equ:multiscale:6}
        \hat{P}_n \hat{w}_{n+1}=\hat{Q}_n \hat{w}_n+\hat{b}_n,
    \end{aligned}
\end{equation}
in which $\lambda=\frac{\tau}{h}<h^{1-1/\beta}$, and
\begin{gather*}
    \hat{P}_n=\begin{bmatrix}
        I & \frac{\sqrt{\tau A_n}}{2h}M_h\\
        \frac{\sqrt{\tau A_n}}{2h}M_h & \left(1+\frac{\varepsilon^2}{\tau}\right) I
    \end{bmatrix},\quad
    \hat{Q}_n=\begin{bmatrix}
        I+\frac{\lambda}{2h^{1-1/\beta}}L_h & O\\
        -\frac{\varepsilon^2}{2h}\sqrt{\frac{\tau}{A_n}}M_h & \frac{\varepsilon^2}{\tau} (I+\frac{\lambda}{2h^{1-1/\beta}}L_h)
    \end{bmatrix},\\
    \hat{b}_n=
    \begin{bmatrix}
        \frac{\lambda}{2h^{1-1/\beta}}(b_{1,n+1}-c_{2,n+1})\\
        -\frac{\varepsilon^2}{2h^{2-1/\beta}}\sqrt{\frac{\tau}{A_n}}(b_{2,n}-c_{1,n})-\frac{\sqrt{\tau A_n}}{2h}b_{2,n+1}
    \end{bmatrix}.
\end{gather*}
For convenience in the estimates below, define
\begin{equation*}
    \begin{aligned}
        \tilde{\lambda}:=\frac{\tau}{h^{2-1/\beta}}=\frac{\lambda}{h^{1-1/\beta}},\qquad 0<\tilde{\lambda}<1.
    \end{aligned}
\end{equation*}
More precisely, Appendix Section \ref{section:appendix:telegraph:dissipative} proves that, for every fixed $\beta\in[1,+\infty)$,
\begin{equation*}
    \begin{aligned}
        \max_{0\le n\le N_t}\left(\|u_n-u(t_n)\|_2+\varepsilon\|v_n-v(t_n)\|_2\right)\le C\left(\tau+h^2+h^{1/\beta}\right),
    \end{aligned}
\end{equation*}
see Eq.~(\ref{equ:appendix:dissipative:6}). Hence the added dissipation vanishes as $h\to0$, so $\beta$ changes the pre-asymptotic smoothing and the complexity balance through $\tilde{\lambda}$, but it does not change the final continuum limit.
\subsubsection{Query Complexity Analysis}
\par The query-complexity analysis for the IMEX scheme applied to the multiscale telegraph equation is more involved than for the previous two examples. Because we do not use a fully implicit scheme, Eq.~(\ref{equ:time:3}) requires estimates of both $\lambda_{\min}\left(\frac{\hat{P}_n+\hat{P}_n^\dagger}{2}\right)$ and $\Vert \hat{Q}_n\Vert_2$. Then, $\hat{Q}_n$ can be expressed as
\begin{equation*}
    \begin{aligned}
        \hat{Q}_n=
        \underbrace{\begin{bmatrix}
        I+\frac{\tilde{\lambda}}{2}L_h & O\\
        O & O
        \end{bmatrix}}_{\hat{Q}_n^{(1)}}+
        \underbrace{\begin{bmatrix}
        O & O\\
        -\frac{\varepsilon^2}{2h}\sqrt{\frac{\tau}{A_n}}M_h & \frac{\varepsilon^2}{\tau}(I+\frac{\tilde{\lambda}}{2}L_h)
        \end{bmatrix}}_{\hat{Q}_n^{(2)}}.
    \end{aligned}
\end{equation*}
Assume in addition that $A_n\ge a_*>0$ uniformly in $n$, with $a_*$ independent of $h$ and $\varepsilon$. By Lemma \ref{lemma:appendix:A:1}, we have $\Vert \hat{Q}_n\Vert_2\le \Vert \hat{Q}_n^{(1)}\Vert_2+\Vert \hat{Q}_n^{(2)}\Vert_2$. Using Lemma \ref{lemma:appendix:B:1} and $\tau=\lambda h$, one obtains
\begin{equation*}
    \begin{aligned}
        \Vert \hat{Q}_n^{(1)}\Vert_2&=\left\Vert I+\frac{\tilde{\lambda}}{2}L_h\right\Vert_2
        =1+\frac{\tilde{\lambda}}{2}\left(-2+2\cos\frac{\pi}{N_x+1}\right)\\
        &\lesssim 1-\frac{\tilde{\lambda}\pi^2}{2}h^2+\mathcal{O}(h^4),
    \end{aligned}
\end{equation*}
Moreover, since $\|L_h\|_2\le4$ and $\|M_h\|_2\le2$, there exists an constant independent of $h$ and $\varepsilon$, such that
\begin{equation*}
    \begin{aligned}
        \Vert \hat{Q}_n^{(2)}\Vert_2\le\frac{\varepsilon^2}{2h}\sqrt{\frac{\tau}{a_*}}\|M_h\|_2+\frac{\varepsilon^2}{\tau}\left\|I+\frac{\tilde{\lambda}}{2}L_h\right\|_2=\mathcal{O}\left(\frac{\varepsilon^2}{\tau}\right).
    \end{aligned}
\end{equation*}
Hence, we can obtain an upper bound for $\Vert \hat{Q}_n\Vert_2$ as follows:
\begin{equation}
    \begin{aligned}
        \label{equ:multiscale:7}
        \Vert \hat{Q}_n\Vert_2\le 1-\frac{\tilde{\lambda}\pi^2}{2}h^2+\mathcal{O}\left(\frac{\varepsilon^2}{\tau}\right)+\mathcal{O}(h^4).
    \end{aligned}
\end{equation}
For $\lambda_{\min}\left(\frac{\hat{P}_n+\hat{P}_n^\dagger}{2}\right)$, we use $M_h^\dagger=-M_h$ and the fact that the two off-diagonal coefficients of $\hat{P}_n$ coincide. Hence
\begin{equation}
    \begin{aligned}
        \label{equ:multiscale:8}
        \lambda_{\min}\left(\frac{\hat{P}_n+\hat{P}_n^\dagger}{2}\right)=\lambda_{\min}\left(\begin{bmatrix}
        I & O\\
        O & \left(1+\frac{\varepsilon^2}{\tau}\right) I
    \end{bmatrix}\right)\ge 1.
    \end{aligned}
\end{equation}
Therefore, applying Eq.~(\ref{equ:time:3}) together with Eqs. (\ref{equ:multiscale:7}) and (\ref{equ:multiscale:8}) yields
\begin{equation*}
    \begin{aligned}
        T_{\text{evol}}=\mathcal{O}\left(\frac{\log\delta^{-1}}{\frac{\tilde{\lambda}\pi^2}{2}h^2-\mathcal{O}\left(\frac{\varepsilon^2}{\tau}\right)+\mathcal{O}(h^4)}\right),
    \end{aligned}
\end{equation*}
In particular, if $\tau=\Theta(h^{2-1/\beta})$ (equivalently $\tilde{\lambda}=\Theta(1)$) and $\varepsilon=o(h^2)$, then $\frac{\varepsilon^2}{\tau}=o(h^{2+1/\beta})$, which is higher order than the leading $\mathcal{O}(h^2)$ term, and thus
\begin{equation*}
    \begin{aligned}
        T_{\text{evol}}=\mathcal{O}(\log\delta^{-1}\cdot N_x^2).
    \end{aligned}
\end{equation*}
\par For sparsity, the matrices $\hat{P}_n$ and $\hat{Q}_n$ have constant-order sparsity, i.e., $\mathcal{O}(1)$. Under the above normalization, the entries of $\hat{Q}_n$ and $\hat{b}_n$ are independent of $\varepsilon^{-1}$, while the largest entries of $\hat{P}_n$ are of order $\mathcal{O}(\sqrt{\tau}/h)$. Therefore, when $\tau=\Theta(h^{2-1/\beta})$,
\begin{equation*}
    \begin{aligned}
        \|\mathbf{H}_{\text{schr}}\|_{\max}=\mathcal{O}\left(\frac{\sqrt{\tau}}{h}\log\delta^{-1}\right)=\mathcal{O}(N_x^{1/(2\beta)}\log\delta^{-1}).
    \end{aligned}
\end{equation*}
The refined estimate in Appendix Section \ref{section:appendix:telegraph:complexity} shows that $\|\text{diag}(\mathbf{F})^\dagger u_{1,1}\|_2=\mathcal{O}(N_x^{-\frac{3}{2}})$, and hence the normalization parameter can be chosen as $K=\mathcal{O}(N_x^{\frac12})$. Since the auxiliary support indicator in Eq.~(\ref{equ:appendix:tele:1}) satisfies $\|\chi\|_2=\sqrt{4(N_t-1)+2N_x}=\mathcal{O}(N_x^{1-\frac{1}{2\beta}})$, while $\|\mathbf{u}(T)\|_2=\mathcal{O}(\sqrt{N_xN_t})=\mathcal{O}(N_x^{\frac32-\frac{1}{2\beta}})$ when $N_t\sim N_x^{2-1/\beta}$, the associated amplitude-amplification overhead remains $g_K=\mathcal{O}(1)$. Therefore, simulating the Hamiltonian to solve the multiscale telegraph equation in Eq.~(\ref{equ:multiscale:1}) requires $\mathcal{O}(N_x^{2+\frac{1}{2\beta}}\log N_x(\log\delta^{-1})^2)$ ($\mathcal{O}(N_x^{2+\frac{1}{2\beta}}(\log N_x)^3)$ if $N_x=\delta^{-1}$) queries in total in the regime $\varepsilon=o(h^2)$ with $\tau=\Theta(h^{2-1/\beta})$, and this bound remains \textit{independent of} $\varepsilon^{-1}$.

\subsubsection{Numerical Example}

\par Finally, we present numerical simulations for the multiscale telegraph equation in Eq.~(\ref{equ:multiscale:1}) to verify the feasibility of the proposed quantum IMEX scheme. The MATLAB scripts implement a support-compressed realization of the same $\chi$-based homogeneous extension analyzed in Appendix Section \ref{section:appendix:telegraph:complexity}: in the first $N_t-1$ time blocks they retain the boundary entries in both components, while in the last time block they retain the full $2N_x$ source block. In other words, the auxiliary sector is compressed exactly to the support of $\chi$. Following this implementation, the telegraph data at each time level are represented by the rescaled vector $\hat{w}_n=[u_n;\sqrt{\tau/A_n}\,v_n]$, where $A_n=a_n-\varepsilon^2$, and the physical flux is recovered from the second block by multiplying by $\sqrt{A_n/\tau}$.
\par We consider two values of the scaling parameter, namely $\varepsilon = 10^{-2}$ and $\varepsilon = 10^{-6}$, which correspond in the MATLAB code to the parameter choices $\varepsilon^2=10^{-4}$ and $\varepsilon^2=10^{-12}$, respectively. In both scripts we take $\beta=2$, $T=0.1$, 16 interior spatial unknowns, the coefficient profile $a(t)=0.5t+0.25$, and the normalization parameter $K=\sqrt{N_x}$, in agreement with the appendix estimate. The time step is chosen as $\Delta t = 0.5\,\Delta x^{3/2}$. The results in Fig.~\ref{fig:result:5} show that the recovered mass density $u$ and mass flux $v$ agree well with the classical discrete solution.

\begin{figure}[htbp]
    \centering
    \subfigure[The mass density $u$ for $\varepsilon=10^{-2}$]{\includegraphics[width=0.375\linewidth]{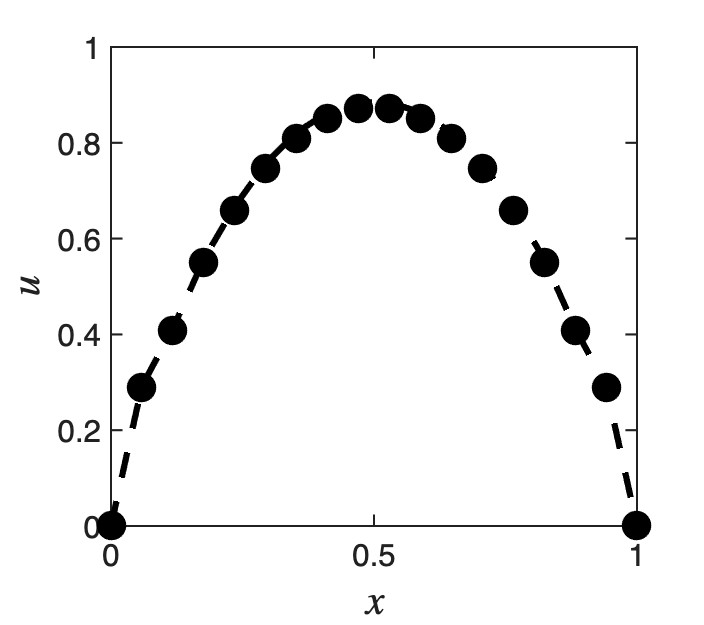}}\qquad
    \subfigure[The mass flux $v$ for $\varepsilon=10^{-2}$]{\includegraphics[width=0.375\linewidth]{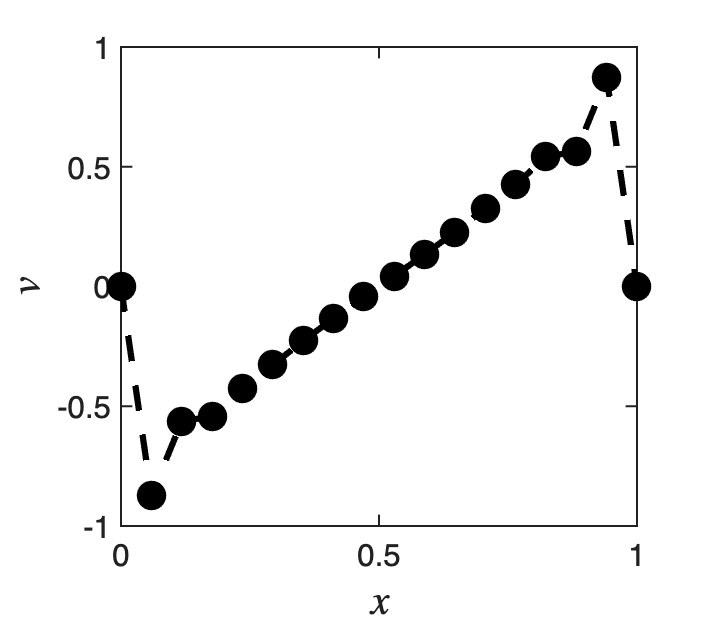}}\\
    \subfigure[The mass density $u$ for $\varepsilon=10^{-6}$]{\includegraphics[width=0.375\linewidth]{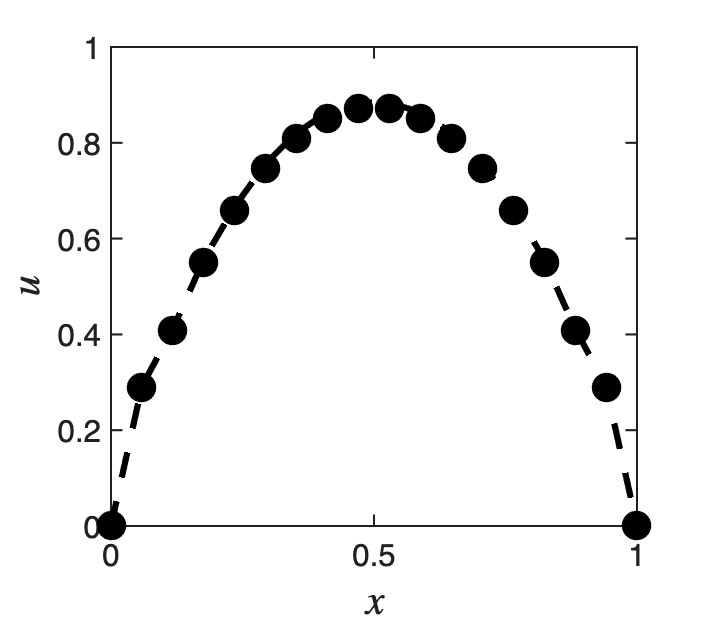}}\qquad
    \subfigure[The mass flux $v$ for $\varepsilon=10^{-6}$]{\includegraphics[width=0.375\linewidth]{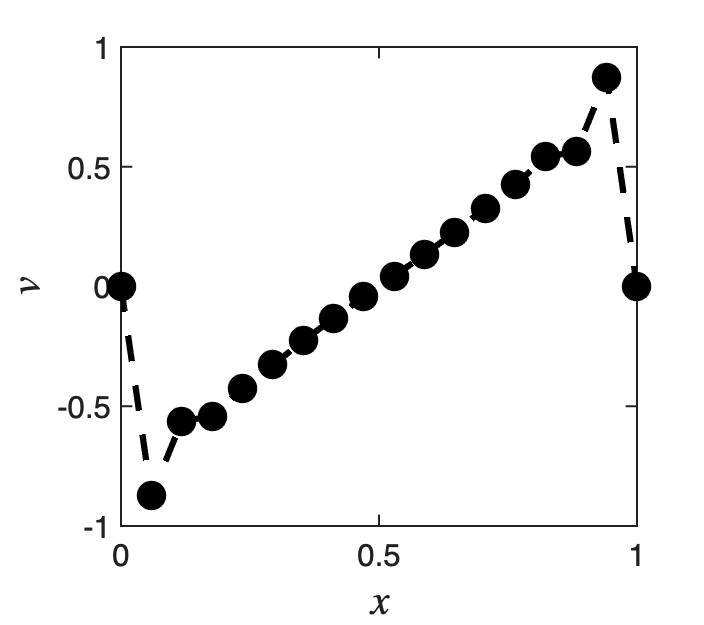}}
    \caption{Numerical results for Eq.~(\ref{equ:multiscale:1}) with $a(t)=0.5t+0.25$ and $\beta=2$ after Schr\"odingerization, shown at $t=0.1$. The MATLAB implementation uses 16 interior spatial unknowns, $\Delta t = 0.5\,\Delta x^{3/2}$, and a support-compressed realization of the $\chi$-based homogeneous extension with $K=\sqrt{N_x}$.}
    \label{fig:result:5}
\end{figure}

\section{Conclusions and Discussions}

\par In this paper, we present a quantum IMEX scheme for multiscale equations whose parameters are independent of the scaling parameter $\varepsilon$. Our approach is based on the Schr\"odingerization framework and uses an "ODE $\to$ Linear System $\to$ Equivalent ODE" procedure, which allows the evolution time $T_{\text{evol}}$ of the equivalent ODE to decouple from the original physical time $T$. This decoupling provides a practical way to handle stiff and multiscale problems.

\par Our framework also requires a careful determination of the evolution-time parameter $T_{\text{evol}}$. To address this issue, we analyze $T_{\text{evol}}$ through the matrix exponential $e^{-\mathbf{H}t}$ and discuss two complementary approaches. The first provides a practical estimate for time-dependent cases based on the logarithmic norm, while the second uses Laplace transforms and their inverses. Together, these analyses provide a workable basis for estimating $T_{\text{evol}}$.

\par Although our framework shows promise for multiscale problems in both time-dependent and time-independent settings, several avenues for further investigation remain. For example, the quantum IMEX framework could be extended to other numerical discretizations, such as IMEX Runge-Kutta methods \cite{Koto2008IMEXRK,Izzo2017RK} and IMEX multistep methods \cite{Hundsdorfer2007IMEX,Chaudhry2015IMEXMS}. Progress in this direction will require additional refinement of the proposed methods and may further broaden the range of Hamiltonian-simulation techniques for scientific computing.
\section*{Code Availability}
\par The code supporting the findings reported in the main text and the supplementary material will be made publicly available upon acceptance.

\section*{Declaration of competing interest}
The authors declare that they have no known competing financial interests or personal relationships that could have influenced the work reported in this paper.

\section*{Acknowledgement}
SJ acknowledges the support of the NSFC grant No. 12341104, the Shanghai Pilot Program for Basic Research, the Science and Technology Commission of Shanghai Municipality (STCSM) grant no. 24LZ1401200, the Shanghai Jiao Tong University 2030 Initiative, and the Fundamental Research Funds for the Central Universities.

\addcontentsline{toc}{section}{References}
\bibliographystyle{unsrt}
\bibliography{reference}
\appendix
\addcontentsline{toc}{section}{Appendix}
\section{Essential Lemmas and Conclusions}
\subsection{Lemma on Block Matrices}
\par To estimate the upper bound of $\Vert e^{-\mathbf{H}t}\Vert_2$, we present the following useful theorem and its proof.
\begin{lemma}
\label{lemma:appendix:A:1}
\par For a block matrix with $m$ blocks and the following structure:
\begin{equation*}
    \begin{aligned}
        A=\begin{bmatrix}
        A_{11} & A_{12} & \cdots & A_{1m} \\
        A_{21} & A_{22} & \cdots & A_{2m} \\
        \vdots & \vdots & \ddots & \vdots \\
        A_{m1} & A_{m2} & \cdots & A_{mm}
        \end{bmatrix},
    \end{aligned}
\end{equation*}
where all $A_{ij}$ are square matrices of the same size. Then the operator 2-norm of this matrix satisfies:
\begin{equation}
    \begin{aligned}
        \label{equ:appendix:A:1}
        \Vert A\Vert_2\le\sum\limits_{k=-(m-1)}^{m-1}\max_{j-i=k}\Vert A_{ij}\Vert_2.
    \end{aligned}
\end{equation}
\begin{proof}
\par First, we decompose $A$ and let $H_k$ be the matrix composed of blocks $A_{i,j}$ that satisfy $j-i=k$, i.e.
\begin{equation*}
    \begin{aligned}
        A&=
        \underbrace{\begin{bmatrix}
        A_{1,1} &&&\\
        & A_{2,2} &&\\
        && \ddots &\\
        &&& A_{m,m}
        \end{bmatrix}}_{H_0}
        +\underbrace{\begin{bmatrix}
        O & A_{1,2} &&\\
        & O & \ddots &\\
        && \ddots & A_{m-1,m} \\
        &&& O
        \end{bmatrix}}_{H_1}
        +\cdots+\underbrace{\begin{bmatrix}
        O & \cdots && A_{1m} \\
        & O & \ddots &\\
        && \ddots & \vdots\\
        &&& O
        \end{bmatrix}}_{H_{m-1}}\\
        &\quad+\underbrace{\begin{bmatrix}
        O &&&\\
        A_{2,1} & O &  &\\
        & \ddots & \ddots &\\
        && A_{m,m-1} & O
        \end{bmatrix}}_{H_{-1}}
        +\cdots+\underbrace{\begin{bmatrix}
        O &&&\\
        \vdots & O & &\\
        & \ddots & \ddots &\\
        A_{m1} && \cdots & O
        \end{bmatrix}}_{H_{-(m-1)}}.
    \end{aligned}
\end{equation*}
First, we can use the triangle inequality for the 2-norm to obtain
\begin{equation*}
    \begin{aligned}
        \Vert A\Vert_2\le\sum\limits_{k=-(m-1)}^{m-1}\Vert H_k\Vert_2,
    \end{aligned}
\end{equation*}
and using the definition of the 2-norm, we can obtain
\begin{equation*}
    \begin{aligned}
        \Vert H_k\Vert_2\le\max_{j-i=k}\Vert A_{ij}\Vert_2.
    \end{aligned}
\end{equation*}
Combining these two inequalities completes the proof.
\end{proof}
\end{lemma}
\par The following theorem is a standard eigenvalue inequality, and we present its statement without proof.
\subsection{Lemmas on Eigenvalues for Specific Matrix}
\par We provide the following result without proof.
\begin{lemma}
\label{lemma:appendix:B:1}
\par The specific matrices $L_h$ and $M_h$ satisfy the following properties.
\begin{itemize}
    \item For $L_h$, the eigenvalues of $L_h$ are $-2+2\cos\left(\frac{\pi k}{N_x+1}\right)$, $k=1,\dots,N_x$. 
    The maximum eigenvalue is $-2+2\cos\left(\frac{\pi}{N_x+1}\right)$ and is less than 0. The minimum eigenvalue is $-2+2\cos\left(\frac{N_x\pi}{N_x+1}\right)$
    and is greater than $-4$.
    \item 
    For $M_h$, one can compute $M_h^2$ as
    \begin{equation*}
        \begin{aligned}
            M_h^2 =
            \begin{bmatrix}
            -1 & 0 & 1 &&\\
            0 & -2 & 0 &&\\
            1 & 0 & -2 && \ddots \\
            &&& \ddots && 1 \\
            && \ddots & & -2 & 0\\
            &&& 1 & 0 & -1
            \end{bmatrix}.
        \end{aligned}
    \end{equation*}
    The eigenvalues of $M_h$ are $-2i\cos\left(\frac{\pi k}{N_x+1}\right)$, $k=1,\dots,N_x$. Consequently, the eigenvalues of $M_h^2$ are $-4\cos^2\left(\frac{\pi k}{N_x+1}\right)$, $k=1,\dots,N_x$, 
    where the maximum eigenvalue of $M_h^2$ is less than or equal to 0, while the minimum eigenvalue is greater than $-4$.
\end{itemize}
\qed
\end{lemma}
\subsection{Weyl Theorem}
\begin{theorem}
\label{lemma:appendix:B:2}
\par \textbf{(Weyl Inequality).} Let $A$ and $B$ be $n\times n$ Hermitian matrices, and let $C=A+B$. Denote the eigenvalues of $A$, $B$, and $C$ in non-decreasing order as: $\lambda_1(A)\le\cdots\le\lambda_n(A)$, $\lambda_1(B)\le\cdots\le\lambda_n(B)$, and $\lambda_1(C) \le\cdots\le\lambda_n(C)$. Then, for each $k=1,\cdots,n$, the following inequalities hold:
\begin{equation*}
    \begin{aligned}
        \lambda_k(A)+\lambda_1(B)\le\lambda_k(C)\le\lambda_k(A)+\lambda_n(B).
    \end{aligned}
\end{equation*}
\qed
\end{theorem}
\subsection{Logarithmic Norm}
\par In particular, we provide an exact expression for $\mu(A)$ under the $2$-norm.
\begin{lemma}
\label{lemma:appendix:D:1}
\par Let $A_1=\frac{A+A^\dagger}{2}$. Then we have
\begin{equation}
    \begin{aligned}
        \label{equ:appendix:C:1}
        \mu(A)=\lambda_{\max}(A_1),
    \end{aligned}
\end{equation}
in which $\mu(A)$ is defined in Lemma \ref{lemma:lognorm}
\begin{proof}
\par To gain deeper insight into $\mu(A)$, we analyze its relationship with the Hermitian part of the matrix. Using the definition of the $2$-norm and assuming $h$ is small, we can ignore the $h^2$ term, yielding:
\begin{equation}
    \begin{aligned}
        \label{equ:appendix:C:3}
        \Vert I+hA\Vert_2^2
        &=\lambda_{\max}(I+h(A+A^\dagger)+h^2A^\dagger A)\\
        &=\left(1+h\lambda_{\max}(A_1)\right)^2+o(h^2).
    \end{aligned}
\end{equation}
Hence, by applying the definition of $\mu(A)$ given in Lemma \ref{lemma:lognorm} to Eq.~(\ref{equ:appendix:C:3}), we derive
\begin{equation*}
    \begin{aligned}
        \mu(A)=\lambda_{\max}\left(A_1\right),
    \end{aligned}
\end{equation*}
in which $A_1=\frac{A+A^\dagger}{2}$. This completes the proof.
\end{proof}
\end{lemma}
\section{Detailed Query Complexity Analysis for Improved Quantum IMEX Schemes}
\label{section:B}
\subsection{Structure of\texorpdfstring{ $e^{-\mathbf{H}t}$}{} Based on the Laplace Transform}
\label{section:laplacetransform}
\par In this section, we use the Laplace transform and its inverse \cite{Hu2024FundamentalPO} to analyze this matrix exponential.
\begin{remark}
\par For a constant matrix $A$, the Laplace transform of the matrix exponential $e^{At}$ is given by:
\begin{equation*}
    \begin{aligned}
        \mathcal{L}[e^{At}](s)=(sI-A)^{-1}.
    \end{aligned}
\end{equation*}
\qed
\end{remark}
\par Thus, our focus is on computing $(sI + \mathbf{H})^{-1}$. Since $\mathbf{H}$ is block upper triangular, we first use the following decomposition to simplify the calculation. Define $A=Q(sI+P)^{-1}$; then
\begin{equation*}
    \begin{aligned}
        (sI+\mathbf{H})^{-1}=
        \begin{bmatrix}
        (sI+P)^{-1}\\
        & (sI+P)^{-1}\\
        && \ddots\\
        &&& (sI+P)^{-1}
        \end{bmatrix}
        \begin{bmatrix}
        I & A & \cdots & A^{N_t-1} \\
        & I & \cdots & A^{N_t-2}\\
        && \ddots & \vdots\\
        &&& I
        \end{bmatrix}.
    \end{aligned}
\end{equation*}
By applying the inverse Laplace transform and noting that $(s+1)^{-(j+1)}A^j$ corresponds to $\frac{t^j}{j!}e^{-t}A^j$, we obtain the explicit expression for $e^{-\mathbf{H}t}$ as:
\begin{equation}
    \begin{aligned}
        \label{equ:laplace:appendix:1}
        e^{-\mathbf{H}t}=
        \left[\begin{array}{cccc}
        \exp(-Pt) & B_{1,2}(t) & \cdots & B_{1,N_t}(t) \\
        & \exp(-Pt) & \cdots & B_{2,N_t}(t) \\
        && \ddots & \vdots \\
        &&& \exp(-Pt) \\
        \end{array}\right],
    \end{aligned}
\end{equation}
in which the block matrix $B_{i,\ell}$ is defined as follows:
\begin{equation*}
    \begin{aligned}
        B_{i,\ell}(t)
        =\mathcal{L}^{-1}\left[(sI+P)^{-1}\cdot\left(Q\cdot(sI+P)^{-1}\right)^{(\ell-i)}\right](t),
        \ \text{where}\ i<\ell.
\end{aligned}
\end{equation*}
Using the convolution property of the Laplace transform, $B_{i,\ell}(t)$ satisfies the following iterative relationship:
\begin{equation}
    \begin{aligned}
        \label{equ:laplace:2}
        B_{i,\ell}(t)&=\int_0^t\exp(-P\tau)\cdot Q\cdot B_{i,\ell-1}(t-\tau)\d\tau,\\
        B_{i,i}(t)&=\exp(-Pt).
    \end{aligned}
\end{equation}
\subsection{Upper Bound for\texorpdfstring{ $\Vert e^{-\mathbf{H}t}\Vert$}{}}
\par We can solve the iterative relation in Eq.~(\ref{equ:laplace:2}) directly without imposing additional conditions such as the commutativity of $P$ and $Q$. We use mathematical induction to prove that $B_{i,i+k}(t) = e^{-Pt}\cdot \frac{1}{k!} \mathcal{T} \Psi(t)^k$, where $\mathcal{T}$ denotes the time-ordering operator, $\Psi(t) = \int_0^t \tilde{Q}(\tau) d\tau$, and $\tilde{Q}(t)=e^{Pt} Q e^{-Pt}$. The result is immediate for $k=0$. Assuming it holds for $k$, we obtain for $k+1$:
\begin{equation}
    \begin{aligned}
        \label{equ:norm:2}
        B_{i,i+k+1}(t)&=\int_0^te^{-(t-\tau)P}Qe^{-\tau P}\cdot \frac{1}{k!} \mathcal{T}\left(\int_0^\tau \tilde{Q}(s_1)\d s_1\right)\cdots\left(\int_0^{\tau}\tilde{Q}(s_k)\d s_k\right)d\tau\\
        &=e^{-Pt}\cdot \frac{1}{k!}\int_0^t\tilde{Q}(\tau)\cdot \mathcal{T}\left(\int_0^\tau \tilde{Q}(s_1)\d s_1\right)\cdots\left(\int_0^{\tau}\tilde{Q}(s_k)\d s_k\right)d\tau\\
        &=e^{-Pt}\cdot \frac{1}{(k+1)!} \mathcal{T}\left(\int_0^t \tilde{Q}(s_1)\d s_1\right)\cdots\left(\int_0^{t}\tilde{Q}(s_k)\d s_k\right)\left(\int_0^t\tilde{Q}(\tau)d\tau\right)\\
        &=e^{-Pt}\cdot \frac{1}{(k+1)!} \mathcal{T}\Psi(t)^{k+1}.
    \end{aligned}
\end{equation}
Substituting Eq.~(\ref{equ:norm:2}) into Lemma \ref{lemma:appendix:A:1} gives the following result regarding the upper bound of $\Vert e^{-\mathbf{H}t}\Vert_2$:
\begin{equation*}
    \begin{aligned}
        \Vert e^{-\mathbf{H}t}\Vert_2\le \Vert e^{-Pt}\Vert_2\sum\limits_{i=0}^{N_t-1}\frac{\left(\int_0^t \Vert\tilde{Q}(\tau) \Vert_2\d\tau\right)^i}{i!}.
    \end{aligned}
\end{equation*}
One can see that the upper bound of $\Vert e^{-\mathbf{H}t}\Vert_2$ consists of a matrix exponential and a Taylor expansion. Moreover, this expansion can be bounded above by an exponential function. Therefore, we obtain the following simplified upper bound:
\begin{equation}
    \begin{aligned}
        \label{equ:norm:3}
        \Vert e^{-\mathbf{H}t}\Vert_2\le \Vert e^{-Pt+\int_0^t \Vert\tilde{Q}(\tau) \Vert_2\d\tau}\Vert_2\approx\Vert e^{\tau Lt}\Vert_2.
    \end{aligned}
\end{equation}
\noindent The above analysis yields an upper bound for $\Vert e^{-\mathbf{H}t}\Vert_2$ without requiring additional assumptions on $\frac{L+L^{\dagger}}{2}$.
\par These conclusions can be illustrated through a simple example. For Eq.~(\ref{equ:norm:3}), consider $P = I + A + \frac{iB}{\varepsilon}$ and $Q = I$, where $B$ is Hermitian and $\varepsilon \to 0$. Then, using the Trotter formula, one obtains
\begin{equation*}
    \begin{aligned}
        \Vert e^{(-A-\frac{iB}{\varepsilon})t}\Vert_2\le\lim\limits_{n\to\infty}\Vert e^{\frac{-At}{n}}\Vert ^n\Vert e^{-\frac{iB}{n\varepsilon}t}\Vert ^n=\lim\limits_{n\to\infty}\Vert e^{\frac{-At}{n}}\Vert ^n,
    \end{aligned}
\end{equation*}
which indicates that the term involving $\varepsilon^{-1}$ may become negligible in this scaling regime, helping explain the treatment of such multiscale problems, including the highly oscillatory problems considered by Gu \textit{et al.} \cite{Gu2025HighlyOscillatory}.

\section{Detailed Discussion on the Multiscale Telegraph Equation}
\label{section:appendix:telegraph}
\subsection{Vanishing Artificial Dissipation and the Role of \texorpdfstring{$\beta$}{beta}}
\label{section:appendix:telegraph:dissipative}

\par In Eq.~(\ref{equ:multiscale:3}), the parameter $\beta$ appears only through the artificial-dissipation coefficient $\nu_h:=\frac{h^{1/\beta}}{2}$. We now show that this parameter does not alter the final continuum limit. To this end, we introduce the modified dissipative system
\begin{equation}
    \begin{aligned}
        \label{equ:appendix:dissipative:1}
        \partial_tu^{\beta,h}(t,x)+\partial_xv^{\beta,h}(t,x)-\nu_h\partial_{xx}u^{\beta,h}(t,x)&=0,\\
        \partial_tv^{\beta,h}(t,x)+\partial_xu^{\beta,h}(t,x)-\nu_h\partial_{xx}v^{\beta,h}(t,x)&=-\frac{1}{\varepsilon^2}\left(v^{\beta,h}(t,x)+(a(t)-\varepsilon^2)\partial_xu^{\beta,h}(t,x)\right),
    \end{aligned}
\end{equation}
which differs from Eq.~(\ref{equ:multiscale:2}) only through the vanishing-viscosity term $\nu_h\partial_{xx}$. The IMEX scheme in Eq.~(\ref{equ:multiscale:3}) is precisely the first-order-in-time, second-order-in-space discretization of Eq.~(\ref{equ:appendix:dissipative:1}). Hence, if $(u^{\beta,h},v^{\beta,h})$ is sufficiently smooth and the step-size condition $\tilde{\lambda}<1$ holds, then the local truncation error $\mathcal{T}_n$ of Eq.~(\ref{equ:multiscale:3}) relative to Eq.~(\ref{equ:appendix:dissipative:1}) satisfies $\max_{0\le n\le N_t-1}\|\mathcal{T}_n\|_2\le C\left(\tau^2+\tau h^2\right)$, where $C$ depends on the regularity norms of the solution but is independent of $\varepsilon$, $h$, and $\tau$, and a standard discrete Gronwall argument yields the global estimate
\begin{equation}
    \begin{aligned}
        \label{equ:appendix:dissipative:2}
        \max_{0\le n\le N_t}\left(\|u_n-u^{\beta,h}(t_n)\|_2+\varepsilon\|v_n-v^{\beta,h}(t_n)\|_2\right)\le C\left(\tau+h^2\right).
    \end{aligned}
\end{equation}

\par Let $(u,v)$ denote the solution of the original telegraph system in Eq.~(\ref{equ:multiscale:1}), or equivalently Eq.~(\ref{equ:multiscale:2}), under the same initial and boundary data. Define the differences $r:=u^{\beta,h}-u$ and $s:=v^{\beta,h}-v$. Subtracting Eq.~(\ref{equ:multiscale:2}) from Eq.~(\ref{equ:appendix:dissipative:1}) yields
\begin{equation}
    \begin{aligned}
        \label{equ:appendix:dissipative:3}
        \partial_t r+\partial_x s&=\nu_h\partial_{xx}u^{\beta,h},\\
        \partial_t s+\partial_x r&=-\frac{1}{\varepsilon^2}\left(s+(a(t)-\varepsilon^2)\partial_x r\right)+\nu_h\partial_{xx}v^{\beta,h}.
    \end{aligned}
\end{equation}
Assume that $a(t)$ is bounded on $[0,T]$, that the boundary data are compatible, and that the solutions remain uniformly bounded in $H^2$ on $[0,T]$. Multiplying the two equations in Eq.~(\ref{equ:appendix:dissipative:3}) by $r$ and $\varepsilon^2 s$, respectively, integrating over space, and using Young's inequality, one obtains the energy estimate
\begin{equation}
    \begin{aligned}
        \label{equ:appendix:dissipative:5}
        \frac{d}{dt}\left(\|r(t)\|_2^2+\varepsilon^2\|s(t)\|_2^2\right)+\|s(t)\|_2^2\le C\left(\|r(t)\|_2^2+\varepsilon^2\|s(t)\|_2^2\right)+C\nu_h^2,
    \end{aligned}
\end{equation}
where $C$ depends on $T$, the coefficient bounds, and the regularity norms of the solutions, but is independent of $h$ and $\tau$. Gronwall's inequality then yields
\begin{equation}
    \begin{aligned}
        \label{equ:appendix:dissipative:4}
        \sup_{t\in[0,T]}\left(\|u^{\beta,h}(t)-u(t)\|_2+\varepsilon\|v^{\beta,h}(t)-v(t)\|_2\right)\le C\nu_h=C h^{1/\beta}.
    \end{aligned}
\end{equation}
Combining Eqs. (\ref{equ:appendix:dissipative:2}) and (\ref{equ:appendix:dissipative:4}) yields the total error bound
\begin{equation}
    \begin{aligned}
        \label{equ:appendix:dissipative:6}
        \max_{0\le n\le N_t}\left(\|u_n-u(t_n)\|_2+\varepsilon\|v_n-v(t_n)\|_2\right)\le C\left(\tau+h^2+h^{1/\beta}\right).
    \end{aligned}
\end{equation}
Consequently, for every fixed $\beta\in[1,+\infty)$, the artificial dissipation introduced in Eq.~(\ref{equ:multiscale:3}) vanishes as $h\to0$, and the scheme converges to the same solution of the original telegraph equation in Eq.~(\ref{equ:multiscale:1}). The parameter $\beta$ changes only the strength of the pre-asymptotic regularization and the complexity scaling through the choice of $\tau=\Theta(h^{2-1/\beta})$; it does not affect the final continuum limit.

\subsection{Query Complexity Analysis for the Multiscale Telegraph Equation}
\label{section:appendix:telegraph:complexity}

\par In this appendix, we explain why the homogeneous matrix used in the Schr\"odingerization step can still be normalized so that the relevant eigenvalue of its Hermitian part stays bounded from below by a positive quantity that does not deteriorate as $\varepsilon\to0$. For this purpose, we introduce an explicit scaling factor $K$ in the source block and consider
\begin{equation}
    \begin{aligned}
        \label{equ:appendix:tele:1}
        \mathbf{H}_{\text{homo}}(K)&=
        \begin{bmatrix}
            \mathbf{H} & -\text{diag}(\mathbf{F})/K\\
            O & O
        \end{bmatrix},\qquad
        \mathbf{u}_{\text{homo}}(t;K)=
        \begin{bmatrix}
            \mathbf{u}(t)\\
            K\chi
        \end{bmatrix},\\
        &\mathbf{H}_{\text{homo,1}}:=\frac{\mathbf{H}_{\text{homo}}(K)+\mathbf{H}_{\text{homo}}(K)^\dagger}{2}.
    \end{aligned}
\end{equation}
where $\chi:=(\mathbf{1}_{N_t}-e_{N_t})\otimes\begin{bmatrix}e_1+e_{N_x}\\ e_1+e_{N_x}\end{bmatrix}+e_{N_t}\otimes\begin{bmatrix}\mathbf{1}_{N_x}\\ \mathbf{1}_{N_x}\end{bmatrix}$. Here $\mathbf{H}$ and $\mathbf{F}$ are generated by the time-discrete linear system, while $\chi$ records the auxiliary support used in the homogeneous extension: the first $N_t-1$ time blocks keep the boundary entries of both components, and the last block keeps the full source block. Because the physical mode $u_{1,1}$ has vanishing lower component, only the upper density component enters the coupling estimate below even though both components are retained in $\chi$. In particular, $\chi$ is a $\{0,1\}$-valued support indicator and satisfies $\|\chi\|_2^2=4(N_t-1)+2N_x$. The goal is to choose $K$ so that the nontrivial smallest eigenvalue branch of $\mathbf{H}_{\text{homo,1}}$ does not degenerate when $\varepsilon\to0$.

\par Set $A_n=a_n-\varepsilon^2$ and assume that $A_n\ge a_*>0$ uniformly in $n$, where $a_*$ is independent of $h$ and $\varepsilon$. We only keep the leading-order contribution to the Hermitian part in the limit $\varepsilon\to0$. For the rescaled matrices in Eq.~(\ref{equ:multiscale:6}), the skew-symmetry $M_h^\dagger=-M_h$ implies that the off-diagonal terms of $\hat{P}_n$ cancel in the Hermitian part, so
\begin{equation*}
    \begin{aligned}
        \frac{\hat{P}_n+\hat{P}_n^\dagger}{2}=\begin{bmatrix}
            I & O\\
            O & \left(1+\frac{\varepsilon^2}{\tau}\right)I
        \end{bmatrix}=I_{2N_x}+\mathcal{O}\left(\frac{\varepsilon^2}{\tau}\right).
    \end{aligned}
\end{equation*}
Moreover, $\hat{Q}_n$ can be written as a leading operator independent of $a_n$ plus a small remainder. With $\tilde{\lambda}=\tau/h^{2-1/\beta}$, this reads
\begin{equation}
    \begin{aligned}
        \label{equ:appendix:tele:2}
        \hat{Q}_n=\hat{Q}+R_n,\qquad
        \hat{Q}:=\begin{bmatrix}
            I_{N_x}+\frac{\tilde{\lambda}}{2}L_h & O\\
            O & O
        \end{bmatrix},\qquad
        \|R_n\|_2=\mathcal{O}\left(\frac{\varepsilon^2}{\tau}\right),
    \end{aligned}
\end{equation}
where, for simplicity of notation, $N_x$ denotes the size of the reduced spatial block in the limiting operator. Since the same reduced block appears at every time level, the Hermitian part of the corresponding block matrix can be written as
\begin{equation}
    \begin{aligned}
        \label{equ:appendix:tele:3}
        \widetilde{\mathbf{H}}_1:=\frac{\mathbf{H}+\mathbf{H}^\dagger}{2}=I_{N_t}\otimes I_{2N_x}-\frac{1}{2}\mathsf{P}_{N_t}\otimes \hat{Q}+\mathbf{R}_{\varepsilon,h},\qquad \|\mathbf{R}_{\varepsilon,h}\|_2=\mathcal{O}\left(\frac{\varepsilon^2}{\tau}\right),
    \end{aligned}
\end{equation}
in which $\mathsf{P}_{N_t}$ is the adjacency matrix of the path graph. The matrix $I_{N_t}\otimes I_{2N_x}-\frac{1}{2}\mathsf{P}_{N_t}\otimes \hat{Q}$ is real symmetric, so its spectrum can be computed explicitly from the tensor-product structure, while $\mathbf{R}_{\varepsilon,h}$ can be handled perturbatively.

\par The eigenpairs $\mathsf{P}_{N_t}\xi^{(r)}=\theta_r\xi^{(r)}$ of $\mathsf{P}_{N_t}$ are given by $\theta_r=2\cos\frac{r\pi}{N_t+1}$ and $\xi^{(r)}_j=\sqrt{\frac{2}{N_t+1}}\sin\frac{jr\pi}{N_t+1}$, where $r=1,\cdots,N_t$ and $j=1,\cdots,N_t$. On the other hand, by Lemma \ref{lemma:appendix:B:1}, the eigenpairs $L_h\zeta^{(s)}=\ell_s\zeta^{(s)}$ of $L_h$ are
$\ell_s=-2+2\cos\frac{s\pi}{N_x+1}$ and $\zeta^{(s)}_m=\sqrt{\frac{2}{N_x+1}}\sin\frac{ms\pi}{N_x+1}$, where $s=1,\cdots,N_x$ and $m=1,\cdots,N_x$. Therefore,
\begin{equation}
    \begin{aligned}
        \label{equ:appendix:tele:4}
        \hat{Q}\begin{bmatrix}
            \zeta^{(s)}\\
            0
        \end{bmatrix}=q_s\begin{bmatrix}
            \zeta^{(s)}\\
            0
        \end{bmatrix},\qquad
        q_s=1+\frac{\tilde{\lambda}}{2}\ell_s=1-\tilde{\lambda}\left(1-\cos\frac{s\pi}{N_x+1}\right).
    \end{aligned}
\end{equation}
Combining Eqs. (\ref{equ:appendix:tele:3}) and (\ref{equ:appendix:tele:4}), one sees that the vectors
\begin{equation*}
    \begin{aligned}
        u_{r,s}:=\xi^{(r)}\otimes\begin{bmatrix}
            \zeta^{(s)}\\
            0
        \end{bmatrix}=\sqrt{\frac{2}{N_x+1}}\sqrt{\frac{2}{N_t+1}}\left[\sin\frac{j r\pi}{N_t+1}\right]_{j=1,\cdots,N_t}\otimes
        \begin{bmatrix}
            \left[\sin\frac{ms\pi}{N_x+1}\right]_{m=1,\cdots,N_x}\\
            0
        \end{bmatrix},
    \end{aligned}
\end{equation*}
are eigenvectors of the unperturbed Hermitian part $\widetilde{\mathbf{H}}_1^{(0)}:=I_{N_t}\otimes I_{2N_x}-\frac{1}{2}\mathsf{P}_{N_t}\otimes \hat{Q}$, and the corresponding eigenvalues are $1-\frac{1}{2}\theta_r q_s$. Since $\theta_r$ and $q_s$ attain their maximal values at $r=1$ and $s=1$, respectively, the smallest eigenvalue on the physical branch of $\widetilde{\mathbf{H}}_1^{(0)}$ is
\begin{equation}
    \begin{aligned}
        \label{equ:appendix:tele:6}
        \lambda_{\min}(\widetilde{\mathbf{H}}_1)&=1-\cos\frac{\pi}{N_t+1}\left(1-\tilde{\lambda}+\tilde{\lambda}\cos\frac{\pi}{N_x+1}\right)+\mathcal{O}\left(\frac{\varepsilon^2}{\tau}\right)\sim \frac{\tilde{\lambda}\pi^2}{2N_x^2},
    \end{aligned}
\end{equation}
where the final asymptotic uses $\tau=\Theta(h^{2-1/\beta})$ (so $\tilde{\lambda}=\Theta(1)$) and $\varepsilon=o(h^2)$. The associated normalized eigenvector is still denoted by $u_{1,1}=\xi^{(1)}\otimes\begin{bmatrix}\zeta^{(1)}\\0\end{bmatrix}$. In particular, the physical spectral scale of $\widetilde{\mathbf{H}}_1$ remains independent of the multiscale parameter $\varepsilon$, which is precisely the feature needed in the multiscale regime.

\par Next we study the effect of the source block. Eq.~(\ref{equ:appendix:tele:1}) shows that $\mathbf{H}_{\text{homo,1}}$ is obtained from $\left[\begin{smallmatrix}\widetilde{\mathbf{H}}_1 & 0\\ 0 & 0\end{smallmatrix}\right]$ by adding the off-diagonal perturbation
\begin{equation*}
    \begin{aligned}
        E_K=\begin{bmatrix}
            O & -\text{diag}(\mathbf{F})/(2K)\\
            -\text{diag}(\mathbf{F})^\dagger/(2K) & 0
        \end{bmatrix}.
    \end{aligned}
\end{equation*}
The lower-right zero block now has the same dimension as the active source support, but it still contributes only auxiliary zero modes. Hence the eigenvalue relevant for the decay of the physical variables is the physical branch bifurcating from $\lambda_{\min}(\widetilde{\mathbf{H}}_1)$, rather than the literal smallest eigenvalue of $\mathbf{H}_{\text{homo,1}}$. We now estimate the shift of this branch by standard matrix perturbation. Let $v_*:=[u_{1,1};0]$. Since $E_K$ is purely off-diagonal, the first-order correction vanishes:
\begin{equation}
    \begin{aligned}
        \label{equ:appendix:tele:7}
        v_*^\dagger E_K v_*=0.
    \end{aligned}
\end{equation}
Therefore the first nonzero correction appears at second order. Summing over the auxiliary zero modes gives the Schur-complement correction
\begin{equation}
    \begin{aligned}
        \label{equ:appendix:tele:8}
        \delta\lambda=-\frac{\|\text{diag}(\mathbf{F})^\dagger u_{1,1}\|_2^2}{4K^2\lambda_{\min}(\widetilde{\mathbf{H}}_1)},
    \end{aligned}
\end{equation}
Consequently, using Eqs. (\ref{equ:appendix:tele:7}) and (\ref{equ:appendix:tele:8}), the eigenvalue on this physical branch satisfies
\begin{equation}
    \begin{aligned}
        \label{equ:appendix:tele:9}
        \lambda_{\mathrm{phys}}(\mathbf{H}_{\text{homo,1}})=\lambda_{\min}(\widetilde{\mathbf{H}}_1)-\frac{\|\text{diag}(\mathbf{F})^\dagger u_{1,1}\|_2^2}{4K^2\lambda_{\min}(\widetilde{\mathbf{H}}_1)}+\mathcal{O}(K^{-3}).
    \end{aligned}
\end{equation}
Because $u_{1,1}=\xi^{(1)}\otimes\begin{bmatrix}\zeta^{(1)}\\0\end{bmatrix}$, the coupling strength is measured by the entrywise product $\text{diag}(\mathbf{F})^\dagger u_{1,1}$ on the same time-space mode that minimizes the spectrum of $\widetilde{\mathbf{H}}_1$. Hence a sufficient condition for the correction term in Eq.~(\ref{equ:appendix:tele:9}) to be at most one half of $\lambda_{\min}(\widetilde{\mathbf{H}}_1)$ is
\begin{equation}
    \begin{aligned}
        \label{equ:appendix:tele:10}
        K\ge \frac{\|\text{diag}(\mathbf{F})^\dagger u_{1,1}\|_2}{\sqrt{2}\lambda_{\min}(\widetilde{\mathbf{H}}_1)}.
    \end{aligned}
\end{equation}
The reduction of Eq.~(\ref{equ:appendix:tele:10}) to a pure power of $N_x$ requires additional structural information on $\mathbf{F}$, so we keep the bound in this explicit form. Under this choice, $\lambda_{\mathrm{phys}}(\mathbf{H}_{\text{homo,1}})\ge \frac{1}{2}\lambda_{\min}(\widetilde{\mathbf{H}}_1)$, so the smallest nontrivial eigenvalue on the physical branch of $\mathbf{H}_{\text{homo,1}}$ remains of the same order as that of $\widetilde{\mathbf{H}}_1$. Therefore, after dividing the source block by a normalization factor $K$ satisfying Eq.~(\ref{equ:appendix:tele:10}), the Hermitian part of the total matrix used in the Schr\"odingerization procedure still has a positive lower spectral scale along the physical branch that is independent of the multiscale parameter $\varepsilon$. Since the ODE matrix entering the Schr\"odingerization step is $-\mathbf{H}_{\text{homo}}(K)$, this implies that the corresponding warped-phase parameter satisfies $p^\diamond=0$ on the physical branch used for reconstruction.
\par Hence the value of $K$ in Eq.~(\ref{equ:appendix:tele:1}) also changes the success probability of the quantum algorithm through the overall 2-norm of $\mathbf{u}_{\text{homo}}(T;K)$. Since $\|\mathbf{u}_{\text{homo}}(T;K)\|_2^2=\|\mathbf{u}(T)\|_2^2+K^2\|\chi\|_2^2$, the probability of obtaining the physical component from the normalized homogeneous state is
\begin{equation*}
    \begin{aligned}
        \mathrm{Pr}(\mathbf{u})=\frac{1}{2}e^{-2p^{\diamond}}\frac{\|\mathbf{u}(T)\|_2^2}{\|\mathbf{u}(T)\|_2^2+K^2\|\chi\|_2^2}.
    \end{aligned}
\end{equation*}
Here $p^{\diamond}=0$ for the physical branch selected above, so the dependence on $K$ is entirely through the ratio $\|\mathbf{u}(T)\|_2^2/(\|\mathbf{u}(T)\|_2^2+K^2\|\chi\|_2^2)$. Accordingly, the repetition count obtained by amplitude amplification is
\begin{equation}
    \begin{aligned}
        \label{equ:appendix:tele:11}
        g=\mathcal{O}\left(\sqrt{1+\frac{K^2\|\chi\|_2^2}{\|\mathbf{u}(T)\|_2^2}}\right)=\mathcal{O}\left(1+\frac{K\|\chi\|_2}{\|\mathbf{u}(T)\|_2}\right).
    \end{aligned}
\end{equation}
Substituting Eq.~(\ref{equ:appendix:tele:10}) into Eq.~(\ref{equ:appendix:tele:11}) yields the corresponding bound
\begin{equation}
    \begin{aligned}
        \label{equ:appendix:tele:12}
        g=\mathcal{O}\left(1+\frac{\|\text{diag}(\mathbf{F})^\dagger u_{1,1}\|_2\|\chi\|_2}{\lambda_{\min}(\widetilde{\mathbf{H}}_1)\|\mathbf{u}(T)\|_2}\right),
    \end{aligned}
\end{equation}
which shows explicitly how the normalization parameter $K$ trades a better lower spectral bound for a lower post-selection probability. To make Eq.~(\ref{equ:appendix:tele:12}) explicit for the multiscale telegraph equation, we use the block structure of $\mathbf{F}$ in Eq.~(\ref{equ:system:2}) and the triangle inequality to obtain
\begin{equation*}
    \begin{aligned}
        \|\text{diag}(\mathbf{F})^\dagger u_{1,1}\|_2
        &\le \left(\sum_{n=0}^{N_t-1}|\xi_{N_t-n}^{(1)}|^2\left\|\text{diag}(\hat{b}_n)\begin{bmatrix}
            \zeta^{(1)}\\
            0
        \end{bmatrix}\right\|_2^2\right)^{1/2}\\
        &\qquad+|\xi_{N_t}^{(1)}|\left\|\text{diag}(\hat{Q}_0\hat{w}_0)\begin{bmatrix}
            \zeta^{(1)}\\
            0
        \end{bmatrix}\right\|_2,
    \end{aligned}
\end{equation*}
where the first term collects all boundary-source blocks, including the $n=0$ contribution. We estimate these two terms separately. For the boundary-source part, Eq.~(\ref{equ:multiscale:6}) shows that, to leading order as $\varepsilon\to0$, the upper block of $\hat{b}_n$ is proportional to $b_{1,n+1}-c_{2,n+1}$ and is supported only at the first and last spatial entries. Since
\begin{equation*}
    \begin{aligned}
        \zeta^{(1)}_1=\zeta^{(1)}_{N_x}=\sqrt{\frac{2}{N_x+1}}\sin\frac{\pi}{N_x+1}=\mathcal{O}(N_x^{-\frac{3}{2}}),
    \end{aligned}
\end{equation*}
and the boundary traces entering $b_{1,n},b_{2,n},c_{1,n},c_{2,n}$ remain of order $\mathcal{O}(1)$, one has
\begin{equation}
    \begin{aligned}
        \label{equ:appendix:tele:13}
        \left(\sum_{n=0}^{N_t-1}|\xi_{N_t-n}^{(1)}|^2\left\|\text{diag}(\hat{b}_n)\begin{bmatrix}
            \zeta^{(1)}\\
            0
        \end{bmatrix}\right\|_2^2\right)^{1/2}
        &\le \mathcal{O}(N_x^{-\frac{3}{2}})\left(\sum_{j=1}^{N_t}|\xi_j^{(1)}|^2\right)^{1/2}\\
        &=\mathcal{O}(N_x^{-\frac{3}{2}}),
    \end{aligned}
\end{equation}
For the initial-data part, the last block of $\mathbf{F}$ sits at the endpoint of the time chain, so
\begin{equation*}
    \begin{aligned}
        \xi_{N_t}^{(1)}=\sqrt{\frac{2}{N_t+1}}\sin\frac{N_t\pi}{N_t+1}=\mathcal{O}(N_t^{-\frac{3}{2}}).
    \end{aligned}
\end{equation*}
Moreover, $\hat{Q}_0$ has only $\mathcal{O}(1)$ nonzero entries of size $\mathcal{O}(1)$ in each row, so if the entries of the initial data are of order $\mathcal{O}(1)$ then each component of $\hat{Q}_0\hat{w}_0$ is also $\mathcal{O}(1)$. Therefore,
\begin{equation}
    \begin{aligned}
        \label{equ:appendix:tele:14}
        |\xi_{N_t}^{(1)}|\left\|\text{diag}(\hat{Q}_0\hat{w}_0)\begin{bmatrix}
            \zeta^{(1)}\\
            0
        \end{bmatrix}\right\|_2
        &\le |\xi_{N_t}^{(1)}|\,\|\hat{Q}_0\hat{w}_0\|_{\infty}\,\left\|\begin{bmatrix}
            \zeta^{(1)}\\
            0
        \end{bmatrix}\right\|_2\\
        &=\mathcal{O}(N_t^{-\frac{3}{2}})
        =\mathcal{O}\left(N_x^{-3+\frac{3}{2\beta}}\right),
    \end{aligned}
\end{equation}
again using $N_t\sim N_x^{2-1/\beta}$. Combining Eqs. (\ref{equ:appendix:tele:13}) and (\ref{equ:appendix:tele:14}) yields $\|\text{diag}(\mathbf{F})^\dagger u_{1,1}\|_2=\mathcal{O}(N_x^{-\frac{3}{2}})$. Together with Eq.~(\ref{equ:appendix:tele:6}), this gives $\lambda_{\min}(\widetilde{\mathbf{H}}_1)=\mathcal{O}(N_x^{-2})$, and hence Eq.~(\ref{equ:appendix:tele:10}) implies $K=\mathcal{O}\left(\frac{\|\text{diag}(\mathbf{F})^\dagger u_{1,1}\|_2}{\lambda_{\min}(\widetilde{\mathbf{H}}_1)}\right)=\mathcal{O}(N_x^{\frac12})$. On the other hand, the support vector in Eq.~(\ref{equ:appendix:tele:1}) satisfies $\|\chi\|_2=\sqrt{4(N_t-1)+2N_x}=\mathcal{O}(N_x^{1-\frac{1}{2\beta}})$, while if each component of the discrete solution remains of order $\mathcal{O}(1)$ then $\|\mathbf{u}(T)\|_2=\mathcal{O}(\sqrt{N_xN_t})=\mathcal{O}(N_x^{\frac32-\frac{1}{2\beta}})$, again because $N_t\sim N_x^{2-1/\beta}$. Therefore the repetition count associated with the normalization parameter $K$ satisfies
\begin{equation}
    \begin{aligned}
        \label{equ:appendix:tele:15}
        g_K=\mathcal{O}\left(1+\frac{K\|\chi\|_2}{\|\mathbf{u}(T)\|_2}\right)=\mathcal{O}(1).
    \end{aligned}
\end{equation}
Thus, in the multiscale telegraph regime considered here, the normalization needed to preserve a positive physical spectral gap does not introduce any additional asymptotic repetition overhead.
\end{document}